\documentclass[10pt,twocolumn,journal]{IEEEtran}
\usepackage{cite}
\usepackage[colorlinks=true,linkcolor=black,citecolor=black,urlcolor=black]{hyperref}
\newcommand{\rid}{{\rm rid}}

\usepackage[mathscr]{eucal}
\usepackage{epsfig,epsf,psfrag}
\usepackage{amssymb,amsmath,amsfonts,latexsym}
\usepackage{graphicx}
\usepackage{epstopdf}
\usepackage{bm,xcolor,url}
\usepackage[caption=false]{subfig} 
\usepackage{fixltx2e}
\usepackage{array}
\usepackage{verbatim}
\usepackage{bm}
\usepackage{algpseudocode, cite}
\usepackage{algorithm}
\usepackage{verbatim}
\usepackage{textcomp}
\usepackage{mathrsfs}
\usepackage{relsize}
\usepackage{subfig}
\usepackage{amsthm}
\usepackage{enumerate}
\usepackage{setspace}
\usepackage{cases}
\usepackage{footnote}
\usepackage{tabularx,tabu}
\usepackage{tablefootnote}
\usepackage{threeparttable}
\usepackage{tikz}
\usetikzlibrary{datavisualization}
\usetikzlibrary{datavisualization.formats.functions}
\usetikzlibrary{decorations.pathreplacing}
\usepackage{textcomp}
\usepackage{multicol}
\usepackage{lipsum}
\usepackage{enumitem}
\usepackage{bbm}
\usepackage{stackengine}
\stackMath 
\catcode`~=11 \def\UrlSpecials{\do\~{\kern -.15em\lower .7ex\hbox{~}\kern .04em}} \catcode`~=13 

\allowdisplaybreaks[3]

\newcommand{\vecz}{\mathbf{0}}

\newcommand{\norm}[1]{\left\Vert#1\right\Vert}
\newcommand{\normt}[1]{\Vert#1\Vert}

\newcommand{\abs}[1]{\left\lvert#1\right\rvert}

\newcommand{\nn}{\nonumber}

\newcommand{\defeq}{\triangleq}

\newcommand{\barbx}{\overline{\bf x}}

\newcommand{\barbz}{\overline{\bf z}}

\newcommand{\calA}{\mathcal{A}}
\newcommand{\calB}{\mathcal{B}}

\newcommand{\calE}{\mathcal{E}}
\newcommand{\calF}{\mathcal{F}}
\newcommand{\calG}{\mathcal{G}}
\newcommand{\calH}{\mathcal{H}}

\newcommand{\calP}{\mathcal{P}}

\newcommand{\calS}{\mathcal{S}}
\newcommand{\calT}{\mathcal{T}}

\newcommand{\tilcalB}{\widetilde{\calB}}

\newcommand{\ba}{\mathbf{a}}
\newcommand{\bA}{\mathbf{A}}
\newcommand{\bb}{\mathbf{b}}
\newcommand{\bB}{\mathbf{B}}

\newcommand{\bD}{\mathbf{D}}
\newcommand{\be}{\mathbf{e}}

\newcommand{\bg}{\mathbf{g}}

\newcommand{\bH}{\mathbf{H}}

\newcommand{\bI}{\mathbf{I}}

\newcommand{\bL}{\mathbf{L}}

\newcommand{\bM}{\mathbf{M}}

\newcommand{\bp}{\mathbf{p}}

\newcommand{\bQ}{\mathbf{Q}}

\newcommand{\bs}{\mathbf{s}}
\newcommand{\bS}{\mathbf{S}}

\newcommand{\bU}{\mathbf{U}}
\newcommand{\bv}{\mathbf{v}}
\newcommand{\bV}{\mathbf{V}}

\newcommand{\bW}{\mathbf{W}}
\newcommand{\bx}{\mathbf{x}}

\newcommand{\by}{\mathbf{y}}
\newcommand{\bY}{\mathbf{Y}}
\newcommand{\bz}{\mathbf{z}}

\newcommand{\rmc}{\mathrm{c}}

\newcommand{\rmH}{\mathrm{H}}

\newcommand{\bbE}{\mathbb{E}}

\newcommand{\bbN}{\mathbb{N}}

\newcommand{\bbP}{\mathbb{P}}

\newcommand{\bbR}{\mathbb{R}}

\DeclareMathAlphabet{\mathbsf}{OT1}{cmss}{bx}{n}
\DeclareMathAlphabet{\mathssf}{OT1}{cmss}{m}{sl}

\DeclareSymbolFont{bsfletters}{OT1}{cmss}{bx}{n}  
\DeclareSymbolFont{ssfletters}{OT1}{cmss}{m}{n}
\DeclareMathSymbol{\bsfGamma}{0}{bsfletters}{'000}
\DeclareMathSymbol{\ssfGamma}{0}{ssfletters}{'000}
\DeclareMathSymbol{\bsfDelta}{0}{bsfletters}{'001}
\DeclareMathSymbol{\ssfDelta}{0}{ssfletters}{'001}
\DeclareMathSymbol{\bsfTheta}{0}{bsfletters}{'002}
\DeclareMathSymbol{\ssfTheta}{0}{ssfletters}{'002}
\DeclareMathSymbol{\bsfLambda}{0}{bsfletters}{'003}
\DeclareMathSymbol{\ssfLambda}{0}{ssfletters}{'003}
\DeclareMathSymbol{\bsfXi}{0}{bsfletters}{'004}
\DeclareMathSymbol{\ssfXi}{0}{ssfletters}{'004}
\DeclareMathSymbol{\bsfPi}{0}{bsfletters}{'005}
\DeclareMathSymbol{\ssfPi}{0}{ssfletters}{'005}
\DeclareMathSymbol{\bsfSigma}{0}{bsfletters}{'006}
\DeclareMathSymbol{\ssfSigma}{0}{ssfletters}{'006}
\DeclareMathSymbol{\bsfUpsilon}{0}{bsfletters}{'007}
\DeclareMathSymbol{\ssfUpsilon}{0}{ssfletters}{'007}
\DeclareMathSymbol{\bsfPhi}{0}{bsfletters}{'010}
\DeclareMathSymbol{\ssfPhi}{0}{ssfletters}{'010}
\DeclareMathSymbol{\bsfPsi}{0}{bsfletters}{'011}
\DeclareMathSymbol{\ssfPsi}{0}{ssfletters}{'011}
\DeclareMathSymbol{\bsfOmega}{0}{bsfletters}{'012}
\DeclareMathSymbol{\ssfOmega}{0}{ssfletters}{'012}

\newcommand{\tilb}{\widetilde{b}}

\newcommand{\tilub}{\underline{\tilb}}

\newcommand{\tilC}{\widetilde{C}}

\newcommand{\tilbe}{\widetilde{\be}}

\newcommand{\tilf}{\widetilde{f}}

\newcommand{\tilbg}{\widetilde{\bg}}

\newcommand{\tiln}{\widetilde{n}}

\newcommand{\hatbs}{\widehat{\bs}}

\newcommand{\tilbs}{\widetilde{\bs}}

\newcommand{\tilbv}{\widetilde{\bv}}

\newcommand{\tilx}{\widetilde{x}}

\newcommand{\tilbx}{\widetilde{\bx}}

\newcommand{\bari}{\overline{i}}

\newcommand{\barL}{\overline{L}}

\newcommand{\barmu}{\overline{\mu}}
\newcommand{\barrho}{\overline{\rho}}

\newcommand{\bdelta}{\bm{\delta}}

\newcommand{\tgamma}{\widetilde{\gamma}}
\newcommand{\tGamma}{\widetilde{\Gamma}}

\newcommand{\tkappa}{\widetilde{\kappa}}

\newcommand{\trho}{\widetilde{\rho}}

\newcommand{\tbdelta}{\widetilde{\bdelta}}

\newcommand{\ui}{\underline{i}}

\newcommand{\iid}{i.i.d.\ }

\newcommand{\floor}[1]{\lfloor{#1}\rfloor}
\newcommand{\lrangle}[2]{\left\langle{#1},{#2}\right\rangle}

\newcommand{\lea}{\stackrel{\rm(a)}{\le}}
\newcommand{\leb}{\stackrel{\rm(b)}{\le}}
\newcommand{\lec}{\stackrel{\rm(c)}{\le}}
\newcommand{\led}{\stackrel{\rm(d)}{\le}}

\DeclareMathOperator{\tr}{tr}

\DeclareMathOperator{\supp}{supp}

\newtheorem{theorem}{Theorem} 
\newtheorem*{theorem*}{Theorem}
\newtheorem{lemma}{Lemma}

\newtheorem{prop}{Proposition}
\newtheorem{corollary}{Corollary}
\newtheorem{assump}{Assumption}

\theoremstyle{definition}
\newtheorem{definition}{Definition} 

\theoremstyle{remark}
\newtheorem{remark}{Remark}

\newcommand{\qednew}{\nobreak \ifvmode \relax \else
      \ifdim\lastskip<1.5em \hskip-\lastskip
      \hskip1.5em plus0em minus0.5em \fi \nobreak
      \vrule height0.75em width0.5em depth0.25em\fi}

\graphicspath{{figure/}}

\begin{document}

\title{Stochastic L-BFGS: Improved Convergence Rates and Practical Acceleration Strategies}

\author{Renbo~Zhao,~\IEEEmembership{Student~Member,}
		William~B.~Haskell, 
        and~Vincent~Y.~F.~Tan,~\IEEEmembership{Senior~Member}
        


\thanks{An extended abstract of this paper was accepted by UAI 2017~\cite{Zhao_17d}. 
R.~Zhao is with the Department of Electrical and Computer Engineering (ECE), the Department of Industrial and Systems Engineering (ISE) and the Department of Mathematics (Math), National University of Singapore (NUS). W.~B.~Haskell is with the Department of ISE, NUS. V.~Y.~F.~Tan is with the Department ECE and the Department of Math, NUS. R. Zhao and V.~Y.~F.~Tan are supported in part by the NUS Young Investigator Award (grant number R-263-000-B37-133) and an MoE AcRF Tier 1 Grant (R-263-000-C12-112).  W.~B.~Haskell is supported by an MOE Tier I Grant (R-266-000-104-112). }
}


\maketitle

\begin{abstract}
We revisit the stochastic limited-memory BFGS (L-BFGS) algorithm. By proposing a new coordinate transformation framework for the convergence analysis, we prove improved  convergence rates and computational complexities of the stochastic L-BFGS algorithms compared to previous works. 
 In addition, we propose several practical acceleration strategies to speed up the empirical performance of such algorithms. We also provide theoretical analyses for most of the strategies. Experiments on {large-scale} logistic and ridge regression problems demonstrate that our proposed strategies yield significant improvements vis-\`a-vis  {competing  state-of-the-art algorithms}. 
\end{abstract}


\section{Introduction}\label{sec:intro}

We are interested in the following (unconstrained) convex finite-sum minimization problem 
\begin{equation}
\min_{\bx\in\bbR^d} \left[f(\bx)\defeq\frac{1}{n}\sum_{i=1}^n f_i(\bx)\right], \label{eq:problem}
\end{equation}
where $d$ and $n$ denote the ambient dimension of the decision vector and the number of component functions respectively.
Problems in the form of \eqref{eq:problem} play important roles in machine learning and signal processing. 
One important class of such problems is the {\em empirical risk minimization} (ERM) problem, where each $f_i$ assumes the form 
\begin{equation}
f_i(\bx) \defeq \ell(\ba_i^T\bx,b_i)+\lambda R(\bx). \label{eq:ERM}
\end{equation}
In \eqref{eq:ERM}, $\ell:\bbR\times\bbR\to\bbR_+$ denotes a smooth loss function, $R:\bbR^d\to\bbR_+$ a smooth convex regularizer (e.g., Tikhonov), $\lambda\ge 0$ the regularization weight and $\{(\ba_i,b_i)\}_{i=1}^n\subseteq \bbR^{d+1}$ the set of feature-response pairs. Depending on the form of $\ell$ and $R$, many important machine learning problems---such as logistic regression, ridge regression and soft-margin support vector machines---are special cases of ERM.

We focus on the case where both $n$ and $d$ are large, and $f$ is ill-conditioned (i.e., the condition number of $f$ is large).\footnote{In this work, the condition number of a (strongly) convex function refers to that of its Hessian.} In the context of ERM, this means the  {dataset  $\{(\ba_i,b_i)\}_{i=1}^n$ that defines \eqref{eq:problem} is large and the feature vectors $\ba_i$ have high ambient dimension. However, the points} typically belong to a low-dimensional manifold. Such a setting is particularly relevant in the big-data era, due to unprecedented data acquisition abilities. 

When $n$ is large, the computational costs incurred by the batch optimization methods (both first- and second-order) are prohibitive, since in such methods the gradients of all the component functions $\{f_i\}_{i=1}^n$ need to be computed at each iteration. Therefore, stochastic (randomized) optimization  methods have become very popular. At each iteration, only a subset of component functions, rather than all of them, are processed. 
In this way, for a given time budget, much more progress can be made towards global optima compared to a single-step taken for batch methods.
When $d$ is large, Newton- or quasi-Newton-based methods ~
(both batch and stochastic)
 incur both high computational and storage complexities. 
Consequently only first-order and limited-memory quasi-Newton methods (e.g., L-BFGS~\cite{Liu_89}) are practical in this setting. 

\subsection{Related Works}\label{sec:lit}

When both $n$ and $d$ are large, as in our setting, most of research efforts have been devoted to {\em stochastic first-order methods}, which include stochastic gradient descent (SGD)~\cite{Bottou_98,Bottou_04} and its variance-reduced modifications~\cite{Johnson_13,Defazio_14,Reza_15,Schmidt_17}. However, these methods do not make use of the curvature information. This limits their {abilities} to find highly accurate solutions for ill-conditioned problems. 
In order to incorporate the curvature information in the limited-memory setting, recently much progress have been  made toward developing {\em stochastic L-BFGS} algorithm. A partial list of such works includes~\cite{Schra_07,Bordes_09,Sohl_14,Mokh_14,Mokh_15,Mokh_17,Byrd_16a,Moritz_16,Gower_16}. 
In particular, the first convergent 
algorithm was proposed by Mokhtari and Ribeiro~\cite{Mokh_14}.  
Later, the algorithm in~\cite{Byrd_16a} makes use of the subsampled Hessian-vector products to form the correction pairs (as opposed to using difference of stochastic gradients) and achieves better results than previous methods. However, the convergence rate is sublinear (in the strongly-convex case), similar to that of SGD. Later, Moritz~{\em et al}.~\cite{Moritz_16} combines this method with {\em stochastic variance-reduced gradient} (SVRG) and proves linear convergence of the resulting algorithm. The algorithm in~\cite{Gower_16} maintains the structure of this algorithm but incorporates the {\em block BFGS update}  to collect more curvature information in the optimization process. Although the convergence rate of this new method is similar to  that in \cite{Moritz_16}, experimental results  demonstrate practical  speedups introduced by the block BFGS update. Finally, there also exist a large volume of works on decentralized 
second-order methods~\cite{Wei_13,Bert_15,Mokh_16,Mokh_16b,Mokh_17b,Eisen_16,Eisen_17} that aim to coordinate multiple distributed agents (with computational and storage abilities) in the optimization task. 
Since we are not concerned with decentralized optimization algorithms in this paper, we do not discuss these works in details here.

\subsection{Motivations and Main Contributions}\label{sec:contributions}
Our work can be motivated from both {\em theory} and {\em practice}. In terms of theory, although linear convergence (in expectation) has been shown for both algorithms in \cite{Moritz_16} and \cite{Gower_16}, the convergence rates (and hence computational complexities) therein can be potentially further improved. (The analysis method in \cite{Gower_16} mainly follows that in \cite{Moritz_16}, so we treat the analyses in both works in a unified manner.) 
In addition, these results may be strengthened in a probabilistic sense, e.g., from convergence in expectation to convergence in probability or almost surely. 
In terms of practice, in addition to block BFGS update, there may exist several other practical strategies that can potentially further accelerate\footnote{In this work, we refer ``acceleration'' to general strategies that speed up the algorithm, not necessarily the ones based on momentum methods.} the algorithm in \cite{Moritz_16}. 
Based on these two aspects, our work consists of the following main contributions.

1) We propose a {\em coordinate transformation framework} to analyze   stochastic L-BFGS-type algorithms in~\cite{Moritz_16} and \cite{Gower_16}. Our analysis framework yields {a} much improved (linear) convergence rate (both in expectation and almost surely) and computational complexity. 
The essential idea of our method is to unify the analysis of stochastic first-order and second-order methods; as a result, it {\em opens new avenues of designing and analyzing} other variants of  stochastic second-order algorithms based on their first-order counterparts. 

%

2) We conduct a {\em computational complexity analysis} for the stochastic L-BFGS algorithms, which is the first of its kind. 

3) We propose {\em several practical acceleration strategies} to speed up the convergence of the stochastic L-BFGS algorithm in~\cite{Moritz_16}. 
We demonstrate the efficacy of our strategies through numerical experiments on logistic and ridge regression problems. We also prove linear convergence for most of these strategies. 

\section{Preliminaries}

\subsection{Notations}
We use lowercase, bold lowercase and bold uppercase letters to denote scalars, vectors and matrices respectively. 
For a matrix $\bU\in\bbR^{m_1\times m_0}$, we denotes its $(p,q)$-th entry as $u_{pq}$. For a function $f:\bbR^{m_1}\to\bbR^{m_2}$, define  the function $f\circ\bU$ as the composition $(f\circ\bU)(\bz)\defeq f(\bU\bz)$, for any $\bz\in\bbR^{m_0}$. A continuously differentiable function $g:\bbR^d\to\bbR$ is $L$-smooth ($L\!>\!0$) if and only if $\nabla g$ is $L$-Lipschitz on $\bbR^d$.   
We use $\bbN$ to denote the set of natural numbers. For any $n\in\bbN$, we define $[n]\defeq\{1,\ldots,n\}$ and $(n]\defeq\{0,1,\ldots,n\}$. Accordingly, for a sequence of sets $\{\calA_n\}_{n\ge 0}$, define $\calA_{(n]}\defeq\{\calA_0,\ldots,\calA_n\}$, for any $n\in\bbN$. 
As usual, $\liminf_{n\to\infty}\calA_n\defeq \cup_{n\ge 0}\cap_{j\ge n}\calA_j$ and $\limsup_{n\to\infty}\calA_n\defeq \cap_{n\ge 0}\cup_{j\ge n}\calA_j$. For a set $\calA$, denote its complement as $\calA^\rmc$. 
For any sequence $\{x_i\}_{i\ge 0}$, we define $\sum_{i=p}^{q} x_i \defeq 0$ if $p>q$.  
We use $\norm{\cdot}$ to denote both the $\ell_2$ norm of a vector and the spectral norm of a matrix. 
We use $\bB$ and $\bH$ (with subscripts and superscripts) to denote the approximate Hessian and approximate inverse Hessian in L-BFGS algorithms respectively, following the convention in~\cite{Nocedal_06}. 
$\bH$ is also known as the {\em metric matrix}~\cite{Gold_70}. 
In this work, technical lemmas (whose indices begin with `T') will appear in Appendix~\ref{app:tech_lemma}.  

\subsection{Assumptions on Component Functions $f_i$}
\begin{assump}\label{assump:twice_diff}
For each $i\in[n]$, $f_i$ is convex and twice differentiable on $\bbR^d$. For ERM problems \eqref{eq:ERM}, we assume these two properties are satisfied by the loss function $\ell$ in its first argument on $\bbR$ and by the regularizer $R$ on $\bbR^d$. 
\end{assump}

\begin{assump}\label{assump:sc_sm}
For each $i\in[n]$, $f_i$ is $\mu_i$-strongly convex and $L_i$-smooth on $\bbR^d$, where $0<\mu_i\le L_i$. 
\end{assump}

\begin{remark}
Assumptions~\ref{assump:twice_diff} and \ref{assump:sc_sm} {are standard}   in the analysis of both deterministic and stochastic  second-order optimization methods. 
 The strong convexity of $f_i$ in Assumption~\ref{assump:sc_sm} ensures {\em positive} curvature at any point in $\bbR^d$, which in turn guarantees the well-definedness of the BFGS update. As a common practice in the literature~\cite{Byrd_16a,Moritz_16}, this condition can typically be enforced by adding a strongly convex regularizer (e.g., Tikhonov) to $f_i$. 
Due to the strong convexity, \eqref{eq:problem} has a unique solution, denoted as $\bx^*$.
\end{remark}

\section{Algorithm}\label{sec:algo}



We will provide a refined analysis of the optimization algorithm (with some modifications) suggested in~\cite{Moritz_16} and so we recapitulate  it in 
Algorithm~\ref{algo:SLBFGS-nonuniform}. 
This algorithm can be regarded as a judicious combination of SVRG and L-BFGS algorithms. 
 We use $s$  and $t$ to denote  the outer and inner iteration indices respectively and $r$ to denote the index of {metric matrices} $\{\bH_r\}_{r\ge 0}$. 
We also use $\bx_{s,t}$ and $\bx^s$ to denote an inner iterate and outer iterate respectively.   

Each outer iteration $s$ consists of $m$ inner iterations. Before the inner iterations, we first compute a {\em full gradient} $\bg_s \defeq\nabla f(\bx^s)$.  In each inner iteration $(s,t)$, the only modification that we make with respect to (w.r.t.) the original algorithm in~\cite{Moritz_16} is that in computing the stochastic gradient $\bv_{s,t}$, the index set $\calB_{s,t}\subseteq [n]$ of size $b$ is sampled with replacement {\em nonuniformly}~\cite{Xiao_14,Zhao_15}. Specifically, the elements in $\calB_{s,t}$ are sampled \iid from a discrete distribution $P\defeq(p_1,\ldots,p_n)$, such that for any $i\in[n]$, $p_i={L_i}/{\sum_{j=1}^nL_j}$. As will be seen in Lemma~\ref{lem:bound_var}, compared to uniform sampling, nonuniform sampling leads to a better variance bound on the stochastic gradient $\bv_{s,t}$. 
Using $\calB_{s,t}$ and $\nabla f(\bx^s)$, we then compute $\bv_{s,t}$ according to \eqref{eq:var_reduced_grad} in Algorithm~\ref{algo:SLBFGS-nonuniform},
where 
\begin{equation}
\nabla f_{\calB_{s,t}}(\bx) \defeq \frac{1}{b}\sum_{i\in\calB_{s,t}} \frac{1}{np_i}\nabla f_i(\bx), \,\forall\,\bx\in\bbR^d. \label{eq:mb_gradient} 
\end{equation}
This specific way to construct $\bv_{s,t}$ reduces the variance of $\bv_{s,t}$ to zero as $s\to\infty$ (see Lemma~\ref{lem:bound_var}), and serves as a crucial step in the SVRG framework. 
 
Then we compute the search direction $\bH_r\bv_{s,t}$. 
The metric matrix $\bH_r$ serves as an approximate of the inverse Hessian matrix and therefore contains the local curvature information at the recent iterates. Consequently, $\bH_r\bv_{s,t}$ may act as a better descent direction than $\bv_{s,t}$. Since storing $\bH_r$ may incur high storage cost (indeed, $\Theta(d^2)$ space) for large $d$, (stochastic) L-BFGS-type methods compute $\bH_r\bv_{s,t}$ each time from a set of recent {\em correction pairs} $\calH_r$ (that only occupies $\Theta(d)$ space) and $\bv_{s,t}$. 
In this way, the limitation on memory can be overcome. 

 Denote $M\in\bbN$  as the {memory parameter}. We next describe the construction of the set of recent correction pairs $\calH_r\defeq \{(\bs_j,\by_j)\}_{j=r-M'+1}^r$, where $M'\!\defeq\!\min\{r,M\}$. Together with $\bv_{s,t}$, this set 
  will be used to compute the matrix-vector product $\bH_r\bv_{s,t}$.  
Before doing so, in line~\ref{line:barx}, we first compute the averaged past iterates $\{\barbx_r\}_{r\ge 0}$ from $\{\bx_{s,t}\}_{s\ge 0, t\in(m]}$ for every $\Upsilon$ inner iterates, where $\Upsilon\!\in\![m]$. 
Based on $\barbx_{r-1}$ and $\barbx_r$, we compute the most recent {correction pair}  $(\bs_r,\by_r)$ in line~\ref{line:correction_pair}. Following \cite{Byrd_16a}, in computing $\by_r$,  we first sample an index set $\calT_r\subseteq[n]$ of size $b_\rmH$ uniformly without replacement, and then let $\by_r = \nabla^2 f_{\calT_r} (\barbx_r)\bs_r$, where 
\begin{equation}
\nabla^2 f_{\calT_r} (\barbx_r)\defeq \frac{1}{b_\rmH}\sum_{i\in\calT_r} \nabla^2 f_{i} (\barbx_r),\label{eq:subsampHessian}
\end{equation}
denotes the sub-sampled Hessian at $\barbx_r$. Finally, we update  $\calH_{r-1}$ to $\calH_r$ by inserting $(\bs_r,\by_r)$ into $\calH_{r-1}$ and deleting $(\bs_{r-M'},\by_{r-M'})$ from it. 


Based on $\calH_r$, a direct approach to compute $\bH_r\bv_{s,t}$ would be computing $\bH_r$ first and then forming the product with $\bv_{s,t}$. Computing $\bH_r$ involves applying $M'$ BFGS updates to 
\begin{equation}
\bH_r^{(r-M')}\defeq \frac{\bs_r^T\by_r}{\norm{\by_r}^2}\bI \label{eq:def_H0}
\end{equation}
using $\{(\bs_j,\by_j)\}_{j=r-M'+1}^r$. For each $k\in\{r-M'+1,\ldots,r\}$, the update is
\begin{equation}
\bH^{(k)}_r = \left(\bI-\frac{\by_k{\bs_k}^T}{{\by_k}^T\bs_k}\right)\bH^{(k-1)}_r\left(\bI-\frac{\bs_k{\by_k}^T}{{\by_k}^T\bs_k}\right) + \frac{\bs_k{\bs_k}^T}{{\by_k}^T\bs_k}.\label{eq:update_H}
\end{equation}
Finally we set $\bH_r = \bH_{r}^{(r)}$. 
Instead of using this direct approach, we adopt the two-loop recursion algorithm~\cite[Algorithm 7.4]{Nocedal_06} to compute $\bH_r\bv_{s,t}$ as a whole. 
This method serves the same purpose as the direct one, but, as we shall see, with  much reduced computation. 

At the end of each outer iteration $s$, the starting point of next outer iteration, $\bx^{s+1}$ is either uniformly sampled (option I) or averaged (option II) from all the past inner iterates $\{\bx_{s,t}\}_{t\in[m]}$. As shown in Theorem~\ref{thm:main}, these two options can be analyzed in a unified manner.

\begin{remark}\label{rmk:est_smooth_param}
Under many scenarios (e.g., the ridge  and logistic regression problems in Section~\ref{sec:experiments}), the smoothness parameters $\{L_i\}_{i=1}^n$ can be accurately estimated. (For ERM problems, these parameters are typically data-dependent.) If in some cases, accurate estimates of these parameters are not available, we can simply employ uniform sampling, which is a special case of our weighted sampling technique. 
\end{remark}

\begin{algorithm}[t]
\caption{Stochastic L-BFGS Algorithm with Nonuniform Mini-batch Sampling} \label{algo:SLBFGS-nonuniform}
\begin{algorithmic}[1]
\State {\bf Input}: Initial decision vector $\bx^0$,  mini-batch sizes $b$ and $b_\rmH$, parameters $m$, $M$ and $\Upsilon$, step-size $\eta$, termination threshold $\epsilon$
\State {\bf Initialize} $s:=0$, $r:=0$, $\barbx_0=\vecz$, $\bH_0:=\bI$, $\calH_0:=\emptyset$
\State {\bf Repeat} 
\State \quad Compute a full gradient $\bg_s \defeq\nabla f(\bx^s)$
\State \quad $\bx_{s,0}:=\bx^s$
\State \quad {\bf for} $t=0,1,\ldots,m-1$
\State \quad\quad Sample a set $\calB_{s,t}$ with size $b$ 
\State \quad\quad Compute a variance-reduced gradient 
\begin{equation}
\bv_{s,t}:=\nabla f_{\calB_{s,t}}(\bx_{s,t}) - \nabla f_{\calB_{s,t}}(\bx^{s})+ \bg_s\label{eq:var_reduced_grad}
\end{equation}
\State \quad \quad Compute $\bH_r\bv_{s,t}$ from $\calH_r$ and $\bv_{s,t}$\label{line:two_loop} 
\State \quad\quad $\bx_{s,t+1}:= \bx_{s,t}-\eta\bH_r\bv_{s,t}$ \label{line:iteration}
\State \quad \quad {\bf if} $sm+t>0$ {\bf and} $(sm+t)\equiv0\mod \Upsilon$ 
\State \quad \quad\quad $r:=r+1$
\State \quad \quad\quad $\barbx_r:=\frac{1}{\Upsilon}\Big(\sum_{l=\max\{0,t-\Upsilon+1\}}^t \bx_{s,l}$\label{line:barx}
\Statex \hspace{3cm} $+\sum_{l=\min\{m+1,t-\Upsilon+m+2\}}^m \bx_{s-1,l}\Big)$ 
\State \quad \quad\quad Sample a set $\calT_r$ with size $b_\rmH$ 
\State \quad \quad\quad $\bs_r:=\barbx_{r}-\barbx_{r-1}$, $\by_r := \nabla^2 f_{\calT_r}(\barbx_r)\bs_r$\label{line:correction_pair}
\State \quad \quad\quad Update $\calH_r:= \{(\bs_j,\by_j)\}_{j=r-M'+1}^r$
\State \quad\quad {\bf end if}
\State \quad {\bf end for}
\State \quad {\bf Option I}: Sample $\tau_s$ uniformly randomly from $[m]$ and set $\bx^{s+1}:=\bx_{s,\tau_s}$\label{line:sample_tau_s}   
\State \quad {\bf Option II}: $\bx^{s+1}:= \frac{1}{m}\sum_{t=1}^m\bx_{s,t}$ \label{line:uniform_ave}
\State \quad $s:=s+1$
\State {\bf Until} $\abs{f(\bx^s)-f(\bx^{s-1})} < \epsilon$ 
\State {\bf Output}: $\bx^s$
\end{algorithmic}
\end{algorithm}

\section{Convergence Analysis}\label{sec:conv_analysis}

\subsection{Definitions}\label{sec:def}

Let $\{\ui_j\}_{j\in[n]}$ and $\{\bari_j\}_{j\in[n]}$ be permutations of $[n]$ such that $\mu_{\min}\defeq\mu_{\ui_1}\le \cdots\le \mu_{\ui_n}$ and $L_{\bari_1}\le \cdots\le L_{\bari_n} \defeq L_{\max}$. Given any $\tiln\in[n]$, define 
\begin{align}
\barmu_{\tiln}\defeq \frac{1}{\tiln}\sum_{j=1}^{\tiln} \mu_{\ui_j} \quad\mbox{and}\quad \barL_{\tiln}\defeq \frac{1}{\tiln}\sum_{j=n-\tiln+1}^n L_{\bari_j}. \label{eq:def_mu_L}
\end{align}
Accordingly, define  
\begin{equation}
\kappa_{\max}\defeq \frac{L_{\max}}{\mu_{\min}}\quad\mbox{and}\quad \kappa_{\tiln}\defeq \frac{\barL_{\tiln}}{\barmu_{\tiln}}. \label{eq:def_kappa} 
\end{equation}
In particular, define $\barmu\defeq \barmu_n$, $\barL\defeq \barL_n$ and $\kappa\defeq\barL/\barmu$. 
Denote the probability space on which the sequence of (random) iterates $\{\bx_{s,t}\}_{s\ge 0,t\in(m]}$ in Algorithm~\ref{algo:SLBFGS-nonuniform} is defined as $(\Omega,\Sigma,\bbP)$, where $\Sigma$ is the Borel $\sigma$-algebra of $\Omega$.  
We also define a filtration $\{\calF_{s,t}\}_{s\ge 0,t\in(m-1]}$ such that $\calF_{s,t}$ contains all the information up to the time $(s,t)$. Formally, $\calF_{s,t}\defeq \sigma(\{\tau_j\}_{j=0}^{s-1}\cup\{\calB_{i,j}\}_{i\in (s-1],j\in(m-1]}\cup\{\calB_{s,j}\}_{j=0}^{t-1}\cup\{\calT_j\}_{j=0}^{\floor{(sm+t)/L}})$, 
where $\sigma\left(\{x_j\}_{j=1}^n\right)$ denotes the $\sigma$-algebra generated by random variables $\{x_j\}_{j=1}^n$. 
Define $\calF_s\defeq \calF_{s,0}$. 

To introduce our coordinate transformation framework, we define some transforms of variables appearing in Algorithm~\ref{algo:SLBFGS-nonuniform}. Specifically, for any $s,t,r\ge 0$, 
define $\tilbx_{s,t,r}\defeq \bH_r^{-1/2}\bx_{s,t}$, $\tilbx^{s,r}\defeq \bH_r^{-1/2}\bx^{s}$, $\tilbx_r^*\defeq \bH_r^{-1/2}\bx^*$ and $\tilbv_{s,t,r}\defeq \bH_r^{1/2}\bv_{s,t}$.\footnote{As will be shown in Lemma~\ref{lem:spectral_bound}, for any $r\ge 0$, $\bH_r\succeq\gamma\bI$ for some $\gamma>0$. Therefore, $\bH_r^{1/2}$ (and $\bH_r^{-1/2}$) are well-defined.} We also define transformed functions $\tilf_{i,r}\defeq f_i\circ\bH_r^{1/2}$ and $\tilf_r\defeq \frac{1}{n}\sum_{i=1}^n \tilf_{i,r}$, for any $i\in[n]$ and $r\ge 0$. 

To state our convergence results, we define the notions of linear convergence and R-linear convergence. 

\begin{definition}\label{def:linear_conv}
A sequence $\{\bx_n\}_{n\ge 0}\!\subseteq\!\bbR^d$ is said to converge to $\barbx\!\in\!\bbR^d$ {\em linearly} (or more precisely, {\em Q-linearly}) with rate $\iota\in(0,1)$ if 
\begin{equation}
\limsup_{n\to\infty} \frac{\normt{\bx_{n+1}-\barbx}}{\normt{\bx_{n}-\barbx}} \le \iota. \label{eqn:qlinear}
\end{equation}
We say $\bx_n\to\barbx$ {\em R-linearly} with rate $\iota'\in(0,1)$ if there exists a nonnegative sequence $\{\varepsilon_n\}_{n\ge 0}$ such that $\normt{\bx_{n}-\barbx}\le \varepsilon_n$ for sufficiently large $n$ and  $\varepsilon_n\to 0$ linearly with rate $\iota'$. 
\end{definition}

\subsection{Preliminary Lemmas}\label{sec:prelim}

From the definitions of transformed variables and functions in Section~\ref{sec:def}, we immediately have the following  lemmas. 

\begin{lemma}\label{lem:tilf_equiv}
For any $s,t,r\ge 0$ and $i\in[n]$, we have
\begin{align}
\tilf_{i,r}(\tilbx_{s,t,r})&=f_i(\bx_{s,t}),\label{eq:tilf}\\
\nabla \tilf_{i,r}(\tilbx_{s,t,r})&= \bH_r^{1/2} \nabla f_i(\bx_{s,t}),\label{eq:grad_tilf}\\
\nabla^2 \tilf_{i,r}(\tilbx_{s,t,r})&= \bH_r^{1/2} \nabla^2 f_i(\bx_{s,t})\bH_r^{1/2}.\label{eq:Hessian_tilf}
\end{align}
\end{lemma}

\begin{lemma}\label{lem:tilf_sc_sm}
If there exist $0<\gamma'\le \Gamma'$ such that $\gamma'\bI\preceq\bH_r\preceq\Gamma'\bI$ for all $r\ge 0$, then for any $i\in[n]$ and $r\ge 0$, $\tilf_{i,r}$ is twice differentiable, $(\mu_i\gamma')$-strongly convex and $(L_i\Gamma')$-smooth on $\bbR^d$. Consequently, $\tilf_r$ is twice differentiable, $(\gamma'\barmu)$-strongly convex and $(\Gamma'\barL)$-smooth on $\bbR^d$.
\end{lemma}

Next we derive two other lemmas that will not only be used in the analysis later, 
but have the potential to be applied to more general problem settings. Specifically, Lemma~\ref{lem:bound_var} can be applied to any stochastic optimization algorithms based on SVRG and Lemma~\ref{lem:spectral_bound} can be applied to any finite-sum minimization algorithms based on L-BFGS methods (not necessarily stochastic in nature). The proofs of {Lemmas}~\ref{lem:bound_var} and \ref{lem:spectral_bound} are deferred to Appendices~\ref{sec:proof_bound_var} and \ref{sec:proof_spec_bound} respectively. 

\begin{lemma}[Variance bound of $\bv_{s,t}$]\label{lem:bound_var}
In Algorithm~\ref{algo:SLBFGS-nonuniform}, we have $\bbE_{\calB_{s,t}}\left[\bv_{s,t}\vert\calF_{s,t}\right] = \nabla f(\bx_{s,t})$ and 
\begin{align}
&\bbE_{\calB_{s,t}}\left[\norm{\bv_{s,t}-\nabla f(\bx_{s,t})}^2\vert\calF_{s,t}\right]\nn\\
&\quad\quad\le \frac{4\barL}{b} \left(f(\bx_{s,t})-f(\bx^*)+f(\bx^s)-f(\bx^*)\right).\label{eq:bound_var}
\end{align}
\end{lemma}
\begin{remark}
In previous works~\cite{Moritz_16,Gower_16}, a uniform mini-batch sampling of $\calB_{s,t}$ was employed, and different variance bounds of $\bv_{s,t}$ were derived.  In \cite[Lemma 6]{Moritz_16}, the bound was 
\begin{equation}
4L_{\max}\left(f(\bx_{s,t})-f(\bx^*)+f(\bx^s)-f(\bx^*)\right).\label{eq:bound_Moritz}
\end{equation}
In \cite[Lemma 2]{Gower_16}, this bound was slightly improved   to 
\begin{align}
&4L_{\max}((f(\bx_{s,t})-f(\bx^*))\nn\\
&\quad\quad\quad\quad+\left(1-{1}/{\kappa_{\max}}\right)(f(\bx^s)-f(\bx^*))).\label{eq:bound_Gower}
\end{align}
However, both of these bounds fail to capture the dependence on the mini-batch size $b$. 
In contrast, in this work we consider a nonuniform mini-batch sampling (with replacement). Due to division by $b$ and $\barL\le L_{\max}$ (indeed in many cases, $\barL\ll L_{\max}$), our bound in \eqref{eq:bound_var} is superior to \eqref{eq:bound_Moritz}  and \eqref{eq:bound_Gower}. 
As will be seen in Theorem~\ref{thm:main}, our better bound \eqref{eq:bound_var} leads to a faster (linear) convergence rate of Algorithm~\ref{algo:SLBFGS-nonuniform}.
\end{remark}

\begin{lemma}[Uniform Spectral Bound of $\{\bH_r\}_{r\ge 0}$]\label{lem:spectral_bound}
The spectra of $\{\bH_r\}_{r\ge 0}$ are uniformly bounded, i.e., for each $r\ge 0$, $\gamma\bI\preceq \bH_r \preceq \Gamma\bI$, where\footnote{We assume $\kappa_{b_\rmH}>1$  for any ${b_\rmH}\in[n]$ since we focus on the setting where $f$ is ill-conditioned, i.e., $\kappa_{b_\rmH}\ge \kappa\gg 1$.
If $\kappa_{b_\rmH}=1$ for some ${b_\rmH}\in[n]$, then $\Gamma=(\kappa_{b_\rmH}^M+M)/\barmu_{b_\rmH}$ and $\gamma$ remains the same. The proof for this case can be straightforwardly adapted from that in Section~\ref{sec:proof_spec_bound}.\label{ft:kappa_gg_1}} 
\begin{equation}
\gamma \defeq \frac{1}{(M+1)\barL_{b_\rmH}} \quad\mbox{and}\quad \Gamma \defeq \frac{\kappa_{b_\rmH}^{M+1}}{\barmu_{b_\rmH}(\kappa_{b_\rmH}-1)}.\label{eq:spec_bound}
\end{equation}
\end{lemma}


\begin{remark}\label{rmk:spectral_bound}
In~\cite{Mokh_15,Byrd_16a,Moritz_16}, the authors make use of a classical technique in~\cite{Liu_89} to derive a different uniform spectral bound of $\{\bH_r\}_{r\ge 0}$. Their technique involves applying $\tr(\cdot)$ and $\det(\cdot)$ recursively to the BFGS update rule
\begin{align}
\bB_r^{(k)} = \bB_r^{(k-1)}-\frac{\bB_r^{(k-1)}\bs_k{\bs_k}^T\bB_r^{(k-1)}}{{\bs_k}^T\bB_r^{(k-1)}\bs_k}+\frac{\by_k\by_k^T}{\bs_k^T\by_k},\label{eq:update_B}
\end{align}
where $\bB_r^{(k)}\defeq (\bH_r^{(k)})^{-1}$ denotes the approximate Hessian matrix at step $k$ in the reconstruction of $\bB_r\defeq(\bH_r)^{-1}$. The lower and upper bounds derived by this technique are 
\begin{align}
\tgamma = \frac{1}{(d+M)L_{\max}}\;\,\mbox{and}\;\, 
\tGamma 
&= (d+M)^{d+M-1}\frac{\kappa_{\max}^{d+M-1}}{\mu_{\min}} \nn
\end{align}
respectively. As will be seen in Proposition~\ref{prop:comp}, the overall computational complexity of Algorithm~\ref{algo:SLBFGS-nonuniform} heavily depends on the estimated (uniform) condition number of $\{\bH_r\}_{r\ge 0}$. Therefore, it is instructive to compute this quantity for both $(\gamma,\Gamma)$ and $(\tgamma,\tGamma)$ as
\begin{align}
\kappa_\rmH &\defeq \frac{\Gamma}{\gamma} = (M+1)\frac{\kappa_{b_\rmH}^{M+2}}{\kappa_{b_\rmH}-1}\approx (M+1)\kappa_{b_\rmH}^{M+1},\label{eq:kappa_ours}\\
\tkappa_\rmH &\defeq \frac{\tGamma}{\tgamma} = (M+d)^{M+d}\kappa_{\max}^{M+d}, \label{eq:kappa_Moritz}
\end{align}
where the approximation in \eqref{eq:kappa_ours} follows from $\kappa_{b_\rmH}\gg 1$ (see Footnote~\ref{ft:kappa_gg_1}). By comparing \eqref{eq:kappa_ours} and \eqref{eq:kappa_Moritz}, we notice our estimate for the condition number of $\{\bH_r\}_{r\ge 0}$, namely $\kappa_\rmH$, is smaller than those in \cite{Byrd_16a} and \cite{Moritz_16}, namely $\tkappa_\rmH$, in several aspects. First, $\kappa_\rmH$ does not grow (exponentially) with the data dimension $d$. Second, $\kappa_\rmH$ depends on $\kappa_{b_\rmH}$, which is usually much smaller than $\kappa_{\max}$. Third, even if we set $d=1$ in \eqref{eq:kappa_Moritz}, the factor $M+1$ in \eqref{eq:kappa_ours} is much smaller than the factor $(M+1)^{M+1}$ in \eqref{eq:kappa_Moritz}. 
As a result, our improved estimate of the condition number of $\{\bH_r\}_{r\ge 0}$ will lead to a much better computational complexity estimate (see Proposition~\ref{prop:comp}). 
\end{remark}

\subsection{Main Results}

Our main convergence results consist of Theorem~\ref{thm:main} and Corollary~\ref{cor:as_conv}, which provide linear convergence guarantees of $f(\bx^s)$ to $f(\bx^*)$ in expectation and almost surely, respectively. 


\begin{theorem}\label{thm:main}
In Algorithm~\ref{algo:SLBFGS-nonuniform}, choose $\eta\!<\!\min\{b/12,1\}/(\Gamma\barL)$ and $m$ sufficiently large. With either option I or II, we have 
\begin{align}
\bbE &\left[f(\bx^s)-f(\bx^*)\right] \le \rho^s\left(f(\bx^0)-f(\bx^*)\right),\;\mbox{where}\label{eq:linear_conv}\\
\rho& = \frac{b}{\gamma\barmu m\eta(b-4\eta\Gamma\barL)}+\frac{4\eta\Gamma\barL}{b-4\eta\Gamma\barL}\left(1+\frac{1}{m}\right) < 1. \label{eq:rho}
\end{align}
\end{theorem}
\begin{proof}
Fix an outer iteration $s$ and consider an inner iteration $t$. Define $r \defeq \floor{(sm+t)/L}$. For brevity, we omit   the dependence of $r$ on $s$ and $t$. The iteration in line~\ref{line:iteration} of Algorithm~\ref{algo:SLBFGS-nonuniform} 
  becomes
\begin{equation}
\tilbx_{s,t+1,r}  = \tilbx_{s,t,r} -\eta \tilbv_{s,t,r}. \label{eq:transform_iterate}
\end{equation}
Define \!$\tbdelta_{s,t,r}\!\defeq\! \tilbv_{s,t,r}\!-\!\nabla \tilf_r(\tilbx_{s,t,r})$. From Lemmas~\ref{lem:tilf_equiv} and \ref{lem:bound_var}, 
\begin{align}
&\hspace{-.2cm}\bbE_{\calB_{s,t}}\left.\left[\tilbv_{s,t,r}\right\vert\calF_{s,t}\right] = \bH_r^{1/2} \nabla f(\bx_{s,t}) = \nabla \tilf_r (\tilbx_{s,t,r})\;\;\mbox{and} \label{eq:exp_nabla_tilf}\\
&\hspace{-.2cm}\bbE_{\calB_{s,t}}\left.\left[\normt{\tbdelta_{s,t,r}}^2\right\vert\calF_{s,t}\right] \le \norm{\bH^{1/2}}^2\frac{4\barL}{b} (f(\bx_{s,t})-f(\bx^*)\nn\\
&\hspace{6cm}+f(\bx^s)-f(\bx^*))\nn\\
&\hspace{-.2cm}\le \frac{4\Gamma\barL}{b}\left(\tilf_r(\tilbx_{s,t,r})-\tilf_r(\tilbx_r^*)+\tilf_{r'}(\tilbx^{s,{r'}})-\tilf_{r'}(\tilbx_{r'}^*)\right),\label{eq:var_nabla_tilf}
\end{align}
where $r' \defeq \floor{sm/L}$. 
Using \eqref{eq:transform_iterate}, we can express the distance between $\tilbx_{s,t+1,r}$ and $\tilbx_r^*$ as
\begin{align}
&\norm{\tilbx_{s,t+1,r}-\tilbx_r^*}^2
= \norm{\tilbx_{s,t,r} - \tilbx_r^*}^2 \nn\\
&\hspace{1.5cm}+ 2\eta\left(\frac{\eta}{2}\norm{\tilbv_{s,t,r}}^2 - \lrangle{\tilbv_{s,t,r}}{\tilbx_{s,t,r} - \tilbx_r^*}\right). \label{eq:expanded_dist}
\end{align}
We can show 
\begin{align}
&\frac{\eta}{2}\norm{\tilbv_{s,t,r}}^2 \!-\! \lrangle{\tilbv_{s,t,r}}{\tilbx_{s,t,r} \!-\! \tilbx_r^*}\le \!-\!\left(\tilf_r (\tilbx_{s,t+1,r})\!-\!\tilf_r (\tilbx_{r}^*)\right)\nn\\
&-\lrangle{\tbdelta_{s,t,r}}{\tilbx_{s,t+1,r}-\tilbx_r^*} - \frac{\gamma\barmu}{2}\norm{\tilbx_{s,t,r}-\tilbx_r^*}^2 \label{eq:bound_expanded}
\end{align}
from steps \eqref{eq:first_long} to \eqref{eq:last_long} on the next page. 
\begin{figure*}[h!]
\begin{align} 
&\tilf_r (\tilbx_{s,t+1,r})+\frac{\eta}{2}\norm{\tilbv_{s,t,r}}^2 - \lrangle{\tilbv_{s,t,r}}{\tilbx_{s,t,r} - \tilbx_r^*}+\lrangle{\tbdelta_{s,t,r}}{\tilbx_{s,t+1,r}-\tilbx_r^*} +\frac{\gamma\barmu}{2}\norm{\tilbx_{s,t,r}-\tilbx_r^*}^2\label{eq:first_long}\\ 
=\;& \tilf_r (\tilbx_{s,t+1,r}) - \lrangle{\tilbv_{s,t,r}}{\tilbx_{s,t+1,r} +\frac{\eta}{2}\tilbv_{s,t,r}-\tilbx_r^*}+\lrangle{\tbdelta_{s,t,r}}{\tilbx_{s,t+1,r}-\tilbx_r^*} +\frac{\gamma\barmu}{2}\norm{\tilbx_{s,t,r}-\tilbx_r^*}^2\\
\le\;& \tilf_r (\tilbx_{s,t+1,r}) - \frac{\Gamma\barL}{2}\eta^2\norm{\tilbv_{s,t,r}}^2- \lrangle{\tilbv_{s,t,r}-\tbdelta_{s,t,r}}{\tilbx_{s,t+1,r} -\tilbx_r^*}+\frac{\gamma\barmu}{2}\norm{\tilbx_{s,t,r}-\tilbx_r^*}^2\label{eq:lea}\\
=\;& \tilf_r (\tilbx_{s,t+1,r}) - \lrangle{\nabla \tilf_r (\tilbx_{s,t,r})}{\tilbx_{s,t+1,r} -\tilbx_{s,t,r}}- \frac{\Gamma\barL}{2}\norm{\eta\tilbv_{s,t,r}}^2-\lrangle{\nabla \tilf_r (\tilbx_{s,t,r})}{\tilbx_{s,t,r}-\tilbx_r^*}+\frac{\gamma\barmu}{2}\norm{\tilbx_{s,t,r}-\tilbx_r^*}^2\\
\le\;&\tilf_r (\tilbx_{s,t,r}) +\lrangle{\nabla \tilf_r (\tilbx_{s,t,r})}{\tilbx_r^*-\tilbx_{s,t,r}}+\frac{\gamma\barmu}{2}\norm{\tilbx_r^*-\tilbx_{s,t,r}}^2
\le\;\tilf_r (\tilbx_{r}^*).\label{eq:last_long}
\end{align}\hrule
\end{figure*}
 In~\eqref{eq:lea}, we use the condition $\eta\!\le\! 1/(\Gamma\barL)$. In~\eqref{eq:last_long}, we use the $(\Gamma\barL)$-smoothness of $\tilf_r$ and the $(\gamma\barmu)$-strong convexity of $\tilf_r$ in Lemma~\ref{lem:tilf_sc_sm} respectively.\\
Now, substituting \eqref{eq:bound_expanded} into \eqref{eq:expanded_dist}, we have
\begin{align}
&\norm{\tilbx_{s,t+1,r}-\tilbx_r^*}^2 \le (1-\eta\gamma\barmu)\norm{\tilbx_{s,t,r} - \tilbx_r^*}^2 \nn\\
&- 2\eta\left(\tilf_r (\tilbx_{s,t+1,r})-\tilf_r (\tilbx_{r}^*)\right)-2\eta\lrangle{\tbdelta_{s,t,r}}{\tilbx_{s,t+1,r}-\tilbx_r^*}\nn\\
&=(1-\eta\gamma\barmu)\norm{\tilbx_{s,t,r} - \tilbx_r^*}^2 - 2\eta\left(\tilf_r (\tilbx_{s,t+1,r})-\tilf_r (\tilbx_{r}^*)\right)\nn\\
&\hspace{1.5cm}+ {2\eta^2} \normt{\tbdelta_{s,t,r}}^2 -2\eta\lrangle{\tbdelta_{s,t,r}}{\tilbx_{s,t,r}-\tilbx_r^*}.\label{eq:succ_bound_sqEuc}
\end{align} 
Taking expectation w.r.t.\ $\calB_{s,t}$ and using \eqref{eq:exp_nabla_tilf} and \eqref{eq:var_nabla_tilf}, we have
\begin{align}
&\left.\!\!\bbE_{\calB_{s,t}}\!\!\left[\norm{\tilbx_{s,t+1,r}-\tilbx_r^*}^2\!\!+\!2\eta(\tilf_r (\tilbx_{s,t+1,r})-\tilf_r (\tilbx_{r}^*))\right\vert\!\calF_{s,t}\right]\nn\\
&\le(1-\eta\gamma\barmu)\norm{\tilbx_{s,t,r} - \tilbx_r^*}^2 +\frac{8}{b}\Gamma\barL\eta^2(\tilf_r(\tilbx_{s,t,r})-\tilf_r(\tilbx_r^*)\nn\\
&\hspace{4.cm}+\tilf_{r'}(\tilbx^{s,{r'}})-\tilf_{r'}(\tilbx_{r'}^*)).  \label{eq:succ_dist_bound}
\end{align}
By bounding the factor $1-\eta\gamma\barmu$ by $1$, we can telescope~\eqref{eq:succ_dist_bound} over $t=0,\ldots,m-1$ and obtain 
\begin{align}
&\left.\bbE_{\calB_{s,(m-1]}}\left[\norm{\tilbx_{s,m,r}-\tilbx_r^*}^2\right\vert\calF_{s}\right]+ 2m\eta\left(1-\frac{4}{b}\Gamma\barL\eta\right)\nn\\ 
& \hspace{1.2cm}\times\Big\{\frac{1}{m}\sum_{t=1}^m\bbE_{\calB_{s,(t-1]}}\left.\left[\tilf_r (\tilbx_{s,t,r})-\tilf_r (\tilbx_r^*)\right\vert\calF_{s,t-1}\right]\Big\}\nn\\
&\hspace{0cm}\le \normt{\tilbx^{s,r'}-\tilbx_{r'}^*}^2+\frac{8}{b}\Gamma\barL\eta^2(1+m)\left(\tilf_{r'}(\tilbx^{s,r'})-\tilf_{r'}(\tilbx_{r'}^*)\right).\label{eq:telescoped_uniform}
\end{align}
If we use option II to choose $\bx^{s+1}$ (in line~\ref{line:uniform_ave}), we have $\tilbx^{s+1,r''}=1/m\sum_{i=1}^m \tilbx_{s,t,r''}$, where $r''\defeq \floor{(s+1)m/L}$. Using \eqref{eq:tilf} and Jensen's inequality, we have
\begin{align}
&\frac{1}{m}\sum_{t=1}^m\bbE_{\calB_{s,(t-1]}}\left.\left[\tilf_r (\tilbx_{s,t,r})-\tilf_r (\tilbx_r^*)\right\vert\calF_{s,t-1}\right] \nn\\
&\hspace{1cm}\ge \bbE_{\calB_{s,(m-1]}}\left.\left[\tilf_{r''}(\tilbx^{s+1,r''})-\tilf_{r''}(\tilbx^*_{r''})\right\vert\calF_{s}\right]. \label{eq:exp_Jensen}
\end{align}
Alternatively, if we use option I to determine $\bx^{s+1}$ (in line~\ref{line:sample_tau_s}), we still have \eqref{eq:exp_Jensen} (with inequality replaced by equality). 
If we further use 
Lemma~\ref{lem:sc_sm} to upper bound the term $\normt{\tilbx^{s,r'}-\tilbx_{r'}^*}^2$ in \eqref{eq:telescoped_uniform}, we have
\begin{align}
&2m\eta\left(1-\frac{4}{b}\Gamma\barL\eta\right)\bbE_{\calB_{s,(m-1]}}\left.\left[\tilf_{r''}(\tilbx^{s+1,r''})-\tilf_{r''}(\tilbx^*_{r''})\right\vert\calF_{s}\right]\nn\\
&\hspace{0cm}\le \left(\frac{8}{b}\Gamma\barL\eta^2(1+m)+\frac{2}{\gamma\barmu}\right)\left(\tilf_{r'}(\tilx^{s,r'})-\tilf_{r'}(\tilbx_{r'}^*)\right).
\end{align}
Using \eqref{eq:tilf} again and rearranging, we have 
\begin{equation}
\bbE\left.\left[f(\bx^{s+1})-f(\bx^*)\right\vert\calF_s\right] \le \rho(f(\bx^s)-f(\bx^*)).
\end{equation}
We take expectation on both sides to complete the proof. 
\end{proof}

\begin{remark}\label{rmk:main}
We compare our linear convergence rate $\rho$ in~\eqref{eq:rho}  with those in \cite{Moritz_16} and \cite{Gower_16}. Since the convergence rates in these two works are almost the same, 
we use the rate in \cite{Moritz_16} for comparison. The linear rate $\trho$ in~\cite{Moritz_16} equals 
\begin{align}
&\frac{1}{2\tgamma\mu_{\min}m\eta(1-\eta\tGamma L_{\max}\kappa_{\max}\tkappa_{\rmH})} + \frac{\eta\tGamma L_{\max}\kappa_{\max}\tkappa_{\rmH}}{1-\eta\tGamma L_{\max}\kappa_{\max}\tkappa_{\rmH}}.\nn
\end{align}
For simplicity, if we let $b=1$, $\mu_{\min}=\barmu$, $L_{\max} = \barL$, $\tgamma=\gamma$, $\tGamma = \Gamma$ and ignore other constant factors,\footnote{However, bear in mind that by doing all these substitutions, $\trho$ has already been much improved.} we notice that there is an additional multiplicative factor $\kappa\kappa_\rmH$ associated with $\eta\Gamma\barL$ in $\trho$. 
As a result $\trho>\rho$. 
A more direct way to observe the detrimental effects of this additional $\kappa\kappa_\rmH$ is to compare the computational complexities resulting from $\rho$ and $\trho$. See Remark~\ref{rmk:complexity} for details.  
The reason that we manage to avoid this factor in our rate $\rho$ is precisely because we adopt the {\em coordinate transformation framework} in our analysis (see proof above). Specifically, by absorbing the sequence of metric matrices $\{\bH_r\}_{r\ge 0}$ into decision vectors and functions, we are able to proceed through bounding the (expected squared Euclidean) distance between $\tilbx_{s,t+1,r}$ and $\tilbx_r^*$, instead of directly bounding $f(\bx_{s,t+1})$ via the smoothness property of $f$ (cf.~proof of Theorem~7 in \cite{Moritz_16}). Thus in our analysis, we avoid the additional appearance of $\barL$ and $\Gamma$ (which leads to the additional factor $\kappa\kappa_\rmH$). 
\end{remark}

\begin{corollary}\label{cor:as_conv}
In Algorithm~\ref{algo:SLBFGS-nonuniform}, $\{f(\bx^s)\}_{s\ge 0}$ converges to $f(\bx^*)$ R-linearly almost surely with rate $\rho$. 
\end{corollary}

\begin{proof}
Our proof is inspired by~\cite[Corollary~2]{Goldfarb_17} and  mainly  
leverages  the Borel-Cantelli lemma~\cite{Will_91}. 
For any $\epsilon'>0$ and $0\!<\!\delta\!<\!1-\rho$, consider the sequence of events 
$\{\calE_s\}_{s\in\bbN}$ such that
\begin{align}
\calE_s&\defeq\left\{\omega\in\Omega:\frac{f(\bx^s(\omega))-f(\bx^*)}{(\rho+1/\sqrt{s})^s}>\epsilon'\right\}, \,\forall\,s\in\bbN. \label{eq:def_cal_E}
\end{align}
Therefore, 
\begin{align}
\sum_{s=1}^\infty \bbP(\calE_s) \lea &\sum_{s=1}^\infty \frac{\bbE[f(\bx^s)-f(\bx^*)]}{\epsilon'(\rho+1/\sqrt{s})^s}\nn\\
\leb &\frac{f(\bx^0)-f(\bx^*)}{\epsilon'}\sum_{s=1}^\infty \left(\frac{\rho}{\rho+1/\sqrt{s}}\right)^s \nn\\
\lec &  \frac{f(\bx^0)-f(\bx^*)}{\epsilon'}\sum_{s=1}^\infty \exp\left(-\frac{\sqrt{s}}{\rho+1/\sqrt{s}}\right)\nn\\
\led & \frac{f(\bx^0)-f(\bx^*)}{\epsilon'}\sum_{s=1}^\infty \exp\left(-\frac{\sqrt{s}}{\rho+1}\right) < \infty, \label{eq:finite_sum_prob}
\end{align}
where in (a) we use Markov's inequality, in (b) we use~\eqref{eq:linear_conv} in Theorem~\ref{thm:main}, in (c) we use  
\begin{equation}
(1+x)^s\le e^{sx}, \forall\,s\in \bbN\cup\{0\}, \forall\, x>-1, \label{eq:lin_exp_ineq}
\end{equation}
and in (d) we use the fact that $s\ge 1$. 
Thus by the Borel-Cantelli lemma, ${\bbP}(\limsup_{s\to\infty}\!\calE_s)\!=\!0$, or equivalently, 
\begin{equation}
\bbP\left(\liminf_{s\to\infty}\calE_s^\rmc\right)=1.
\end{equation}
The definition of $\calE_s$ in~\eqref{eq:def_cal_E} implies that 
\begin{equation}
\liminf_{s\to\infty}\calE_s^\rmc=\left\{\omega\in\Omega:\limsup_{s\to \infty}\frac{f(\bx^s(\omega))-f(\bx^*)}{(\rho+1/\sqrt{s})^s}\le\epsilon'\right\}.\nn 
\end{equation}
Since $\epsilon'>0$ is arbitrary, we have
\begin{equation}
\bbP\left(\lim_{s\to \infty}\frac{f(\bx^s)-f(\bx^*)}{(\rho+1/\sqrt{s})^s}=0\right)=1
\end{equation}
or equivalently, $f(\bx^s)-f(\bx^*)=o((\rho+1/\sqrt{s})^s)$ almost surely, for any $s\ge 0$. 
For convenience, define a sequence $\{\varpi_s\}_{s\ge 0}$ such that $\varpi_s\defeq (\rho+1/\sqrt{s})^s$, for any $s\ge 0$. By applying \eqref{eq:lin_exp_ineq} to $\varpi_s$, we have $\varpi_s=\rho^{s - \Theta(\sqrt{s})}$ (the implied constant in the $\Theta(\cdot)$ notation is positive). Since $\rho^{ - \Theta(\sqrt{s+1}) }/\rho^{ - \Theta(\sqrt{s }) }\to 1$ as $s\to\infty$, $\lim_{s\to\infty} \varpi_{s+1}/\varpi_{s}=\rho$ as desired. 
\end{proof}

\begin{remark}\label{rmk:comp_Goldfarb}
Note that the analysis techniques in Corollary~\ref{cor:as_conv} can be applied to any stochastic algorithm with linear convergence in expectation, therefore are of independent interest. 
Although a similar result was proved in~\cite[Corollary~2]{Goldfarb_17},  
it is weaker than Corollary~\ref{cor:as_conv}. Specifically, the almost sure linear convergence rate therein is {\em strictly} worse than the corresponding linear rate in expectation. In contrast, by leveraging refined analysis techniques, we show that such a degradation can indeed be avoided. 
\end{remark}


\begin{remark}
By the $\barmu$-strong convexity of $f$ (see Assumption~\ref{assump:sc_sm}), we have $\norm{\bx-\bx^*}^2\!\le\! (2/\barmu)(f(\bx)-f(\bx^*))$, for any $\bx\!\in\!\bbR^d$. Therefore, the linear convergence of $\{f(\bx^s)\}_{s\ge 0}$ to $f(\bx^*)$ in expectation (in Theorem~\ref{thm:main}) also implies the R-linear convergences of $\{\bx^s\}_{s\ge 0}$ to $\bx^*$ in expectation. 
Similarly, we can also derive the almost sure R-linear convergence of $\{\bx^s\}_{s\ge 0}$ to $\bx^*$ from Corollary~\ref{cor:as_conv}. 
\end{remark}

\section{Complexity Analysis}\label{sec:comp_analysis}

In this section we provide a framework for analyzing the computational complexity of the stochastic L-BFGS algorithm in Algorithm~\ref{algo:SLBFGS-nonuniform}. 
Our framework can be easily generalized to other stochastic second-order algorithms, e.g., SQN algorithm in~\cite{Byrd_16a}. To begin with, we make two additional assumptions.

\begin{assump}\label{assump:gradient_O(d)}
For any $\bx\in\bbR^d$ and $i\in[n]$, the gradient $\nabla f_i(\bx)$ can be computed in $O(d)$ operations.\footnote{An operation refers to evaluation of an elementary function, such as addition, multiplication and logarithm.}
\end{assump}

\begin{assump}\label{assump:Hess_vect_O(d)}
For any $\bx,\by\in\bbR^d$ and $i\in[n]$, the Hessian-vector product $\nabla^2 f_i(\bx)\by$ can be computed in $O(d)$ operations.
\end{assump}

\begin{remark}
These two assumptions are naturally satisfied for ERM problems \eqref{eq:ERM} with Tikhonov regularization. For these problems, $R(\bx) = \frac{1}{2}\norm{\bx}^2$ and 
\begin{align}
\nabla  \ell(\ba_i^T\bx,b_i)&= \ell'(\ba_i^T\bx,b_i)\ba_i + \lambda\bx,\label{eq:ERM_first_deriv}\\
\nabla^2  \ell(\ba_i^T\bx,b_i)\by&= \ell''(\ba_i^T\bx,b_i)(\ba_i^T\by)\ba_i+\lambda\by,\label{eq:ERM_sec_deriv}
\end{align}
where $\ell'(\cdot,\cdot)$ and $\ell''(\cdot,\cdot)$ are first and second derivatives of $\ell(\cdot,\cdot)$ w.r.t.\ the first argument. We easily see 
that the right-hand sides of both \eqref{eq:ERM_first_deriv} and \eqref{eq:ERM_sec_deriv} can be computed in $O(d)$ operations. 
\end{remark}

From Algorithm~\ref{algo:SLBFGS-nonuniform}, we observe that its total computational cost $C$ can be split into three parts. The first part $C_1$ involves computing the variance-reduced gradient $\bv_{s,t}$ in \eqref{eq:var_reduced_grad}, the second part $C_2$ involves computing $\bH_r\bv_{s,t}$ (via two-loop recursion) in line~\ref{line:two_loop}, and the third part $C_3$ involves computing the correction pair $(\bs_r,\by_r)$ in line~\ref{line:correction_pair}. 

\begin{prop}\label{prop:comp}
Let Assumptions~\ref{assump:twice_diff} to \ref{assump:Hess_vect_O(d)} hold. In Algorithm~\ref{algo:SLBFGS-nonuniform}, 
\begin{align}
C_1 &= O\left((n+\kappa\kappa_\rmH)d\log\left(1/\epsilon\right)\right),\label{eq:C_1}\\
C_2 &= O\left(\kappa\kappa_\rmH d\log\left(1/\epsilon\right)\right),\label{eq:C_2}\\
C_3 &= O\left(d\log\left(1/\epsilon\right)\right).\label{eq:C_3}
\end{align}
Thus the total computational cost $C\defeq \sum_{i=1}^3 C_i$ equals
\begin{equation}
 O\left(\left(n+\kappa\kappa_\rmH\right)d\log\left(1/\epsilon\right)\right).\label{eq:C}
\end{equation}
\end{prop}
\begin{proof}
{We leverage techniques} that have become standard in the SVRG literature (e.g.,~\cite{Xiao_14}). In~\eqref{eq:rho}, if we choose $\eta = \theta b/(\Gamma\barL)$ for some $0<\theta<1/12$, $m = \theta'\kappa\kappa_\rmH/b$ for some large enough positive constant $\theta'$, and use $1+1/m\le 2$, 
then 
\begin{align}
\rho 
&= \frac{1}{\theta'\theta(1-4\theta)}+\frac{8\theta}{1-4\theta}<1.
\end{align}
As a result, the required number of outer iterations to achieve $\epsilon$-suboptimality is $O(\log(1/\epsilon))$. Thus \eqref{eq:C_1} follows from Assumption~\ref{assump:gradient_O(d)} and that $2mb$ gradients (of component functions) are computed in each inner iteration. If we further choose $M=\Theta(b)$, then \eqref{eq:C_2} follows from the fact that two-loop recursion can be done in $O(Md)$ time~\cite[Chapter~7]{Nocedal_06}. Lastly, if we choose $b_\rmH=\Theta(\Upsilon)$, then we obtain \eqref{eq:C_3} using Assumption~\ref{assump:Hess_vect_O(d)}. 
\end{proof}

\begin{remark}\label{rmk:complexity}
Following a similar argument, we can deduce the total complexity estimate $\tilC$ based on the linear rate $\trho$ (see Remark~\ref{rmk:main}) derived in \cite{Moritz_16} as
\begin{equation}
\tilC=O\left(\left(n+b(\kappa_{\max}\tkappa_\rmH)^2\right)d\log\left({1}/{\epsilon}\right)\right).\label{eq:tilC}
\end{equation}
Compared with $\tilC$, we observe that our complexity estimate $C$ in \eqref{eq:C} is much better, in several aspects. First, the dependence of $C$ on the condition number $\kappa\kappa_\rmH$ is linear, rather than quadratic. The quadratic dependence of $\kappa_{\max}\tkappa_\rmH$ in $\tilC$ is precisely caused by the additional $\kappa_{\max}\tkappa_\rmH$ in $\trho$ (see Remark \ref{rmk:main}). 
Second, $C$ is independent of the mini-batch size $b$. The appearance of $b$ in $\tilC$ is a result of the loose bound on variance of $\bv_{s,t}$ (cf.\ \eqref{eq:bound_Moritz} and \eqref{eq:bound_Gower}). Third, the condition number $\kappa\kappa_\rmH$ in $C$ is much more smaller than $\kappa_{\max}\tkappa_\rmH$ in $\tilC$ for ill-conditioned problems. This is a result of the non-uniform sampling of $\calB_{s,t}$ and our improved bound on the spectra of $\{\bH_r\}_{r\ge 0}$. 
\end{remark}

\begin{remark}
As our coordinate transformation framework unifies the design and analysis of stochastic first- and second-order algorithms, we believe that momentum-based acceleration techniques for stochastic first-order methods~\cite{Lin_15,Allen_Zhu_16} can be applied to Algorithm~\ref{algo:SLBFGS-nonuniform} as well. (Details are left to future work.) In this case, the dependence on $\kappa\kappa_\rmH$ in $C$ may be further improved to $\sqrt{\kappa\kappa_\rmH}$~\cite{Allen_Zhu_16}. 
\end{remark}

\section{Acceleration Strategies}\label{sec:acc_strategy}
In this section, we propose three practical acceleration strategies. 
We follow the notations in Section~\ref{sec:algo} and  Algorithm~\ref{algo:SLBFGS-nonuniform}. 
As will be shown in Section~\ref{sec:perf_acc}, all of these strategies result in faster convergence in practice. For the first and second strategies, we also provide their theoretical analyses in Propositions~\ref{prop:exp_samp_ave} and \ref{prop:subsamp_grad_outer}, respectively. See Appendices~\ref{sec:proof_exp_ave} and \ref{sec:proof_subsamp_grad_outer} for the proofs of these two propositions. 

\subsection{Geometric Sampling/Averaging Scheme}\label{sec:exp_samp_ave}
Instead of 
choosing $\bx^{s+1}$ according to option I or II in Algorithm~\ref{algo:SLBFGS-nonuniform}, inspired by \cite{Konecny_13}, 
we can introduce a ``forgetting'' effect by considering two other schemes:\\[.1cm]
{\bf option III}: Sample $\tau_s$ randomly from $[m]$ from the  distribution $Q\defeq (\beta^{m-1}/c,\beta^{m-2}/c,\ldots,1/c)$ and set $\bx^{s+1}:=\bx_{s,\tau_s}$,\\[.1cm]
{\bf option IV}: $\bx^{s+1} := \frac{1}{c}\sum_{t=1}^m{\beta^{m-t}}\bx_{s,t}$,\\[.1cm]
where $0<\beta\le 1-\eta\gamma\barmu<1$ and the normalization constant $c\defeq \sum_{t=1}^m{\beta^{m-t}}$. Since $\beta\in(0,1)$, we observe that in both options III and IV,  
more recent iterates (i.e., iterates $\bx_{s,t}$ with larger indices $t$) will have larger contributions to $\bx^{s+1}$. 
Theoretically, these two schemes can be analyzed in a unified manner, as shown in the following proposition. 

\begin{prop}\label{prop:exp_samp_ave}
In Algorithm~\ref{algo:SLBFGS-nonuniform}, choose $\eta\!<\!\min\{b/12,\!1\}/(\Gamma\barL)$ and $m$ sufficiently large. With either option III or IV, we have
\begin{equation}
\bbE\left[f(\bx^s)-f(\bx^*)\right] \le \barrho^s\left(f(\bx^0)-f(\bx^*)\right),\;\mbox{where}\vspace{-.4cm}\label{eq:faster_linear_conv}
\end{equation}
\begin{align}
\barrho \defeq &\frac{b}{\gamma\barmu c'\eta\left(b-4\eta\Gamma\barL/(1-\eta\gamma\barmu)\right)}\nn\\
&\quad\quad+\frac{4\eta\Gamma\barL}{b-4\eta\Gamma\barL/(1-\eta\gamma\barmu)}\left(1+\frac{1}{c'}\right) < 1 \label{eq:barrho}
\end{align}
and $c'\defeq c/(1-\eta\gamma\barmu)^m$. 
\end{prop}

\begin{remark}\label{rmk:exp_samp_ave}
In the literature~\cite{Moritz_16,Gower_16}, usually option I or~II (in Algorithm~\ref{algo:SLBFGS-nonuniform}) is analyzed   to prove that the stochastic L-BFGS algorithms therein converge linearly. However, for faster convergence in practice, $\bx^{s+1}$ is chosen to be the last inner iterate $\bx_{s,m}$. 
 However, the latter strategy is not amenable to linear convergence analysis. This {\em gap between theory and practice} is filled in by our geometric sampling or averaging scheme, i.e., option III or IV. Specifically, as shown in Figure~\ref{fig:geom_outer_itr}, our scheme not only yields linear convergence in theory, but also performs as well as the ``last inner iterate'' scheme in practice. 
\end{remark}

\subsection{Subsampled Gradient Stabilization}\label{sec:subsamp_grad_outer}

In Algorithm~\ref{algo:SLBFGS-nonuniform}, at the beginning of each outer iteration indexed by~$s$, we compute a full gradient $\bg_s$ to stabilize the subsequent (inner) iterations. Inspired by \cite{Reza_15}, we propose a strategy that only computes a subsampled gradient $\tilbg_s$ at the start of each outer iteration $s$. 
Specifically, we uniformly sample a subset $\tilcalB_s$ of $[n]$ with size $\tilb_s$ {\em without replacement} and then form $\tilbg_s\defeq ({1}/{\tilb_s})\sum_{i\in\tilcalB_s} \nabla f_i(\bx^s)$. 
The size of $\tilcalB_s$, namely $\tilb_s$, increases with the index $s$ until it reaches $n$. 
By judiciously choosing $\tilb_s$, we can show that the resulting algorithm still enjoys linear convergence with rate $\barrho$, when integrated with the geometric sampling/averaging scheme in Section~\ref{sec:exp_samp_ave}. Before we formally state this result in Proposition~\ref{prop:subsamp_grad_outer}, we first make an assumption in addition  to Assumptions~\ref{assump:twice_diff} and \ref{assump:sc_sm}.

\begin{assump}\label{assump:as_bound}
The inner iterates $\{\bx_{s,t}\}_{s\ge 0,t\in[m]}$ generated by the modified algorithm in Section~\ref{sec:subsamp_grad_outer} are bounded almost surely, i.e., there exists $B<\infty$ such that for any $s\ge 0$ and $t\in[m]$, $\norm{\bx_{s,t}-\bx^*}\le B$. 
\end{assump}


\begin{prop}\label{prop:subsamp_grad_outer}
Let Assumptions~\ref{assump:twice_diff}, \ref{assump:sc_sm} and \ref{assump:as_bound} hold. 
For any $\xi>0$ and $S\in\bbN$, and for any $s\in(S]$, 
if we choose $\tilb_s \ge {\tilub}_s\defeq nS^2\alpha_s/(S^2\alpha_s +(n-1)\xi^2\barrho^{2s})$, where $\alpha_s \defeq {1}/{n}\sum_{i=1}^n \norm{\nabla f_i(\bx^s)}^2$, we have
\begin{align}
&\hspace{-.2cm}\bbE\left[f(\bx^{s})-f(\bx^*)\right] \le  \barrho^s\bigg\{\left(f(\bx^{0})-f(\bx^*)\right)\nn\\
&\hspace{-.2cm}+\left(1+\frac{1}{c'}\right)\frac{\xi b}{b-{4\Gamma\barL\eta}/{(1-\eta\gamma\barmu)}}\left(\kappa_\rmH^{1/2}B+{\eta\Gamma \xi}\right)\bigg\}. \label{eq:lin_conv_subsamp_grad}
\end{align}
\end{prop}

\begin{remark}\label{rmk:part_grad}
Several remarks are in order. First, we remark that assumptions involving almost sure boundedness of iterates (e.g., Assumption~\ref{assump:as_bound}) 
are commonplace in the stochastic optimization literature~\cite{Borkar_08,Hu_09,Reza_15} and are always observed to hold in   experiments. Second, under this assumption, we can show that $\{\alpha_s\}_{s\ge 0}$ are bounded almost surely using the Lipschitz continuity of $\nabla f_i$, for any $i\in[n]$ in Assumption~\ref{assump:sc_sm}. Consequently, there exists $B'<\infty$ such that $\alpha_s\le B'$ for any $s\ge 0$ and hence 
\begin{align}
\tilub_s &\le \frac{nS^2B'}{S^2B' +(n-1)\xi^2\barrho^{2s}} \label{eq:bound_tilbs1}\\
&\le \frac{nS^2B'}{(n-1)\xi^2}\left(\barrho^{-2}\right)^s. \label{eq:bound_tilbs2}
\end{align}
As a sanity check, we observe that \eqref{eq:bound_tilbs1} increases to $n$ as $s\to\infty$. By further upper bounding \eqref{eq:bound_tilbs1} by \eqref{eq:bound_tilbs2}, we obtain a practical rule to select $\tilb_s$. Namely, it suffices to choose $\tilb_s = \min\{\zeta\upsilon^s,n\}$, for some constants $\zeta>0$ and $\upsilon>1$. As shown in Section~\ref{sec:experiments}, this rule works well in practice. Third, for any $\epsilon>0$, we can choose $S\in\bbN$ such that
 our algorithm achieves $\epsilon$-suboptimality, i.e., 
 $\bbE\left[f(\bx^{S})-f(\bx^*)\right]< \epsilon$. 
\end{remark}

\subsection{Low-dimensional Approximate Hessians}\label{sec:low_dim_Hessian}
In additional to high dimensionality and large size, {\em sparsity} is also a typical attribute for modern data, i.e., many feature vectors 
only have a few nonzero entries. For ERM problems~\eqref{eq:ERM} (with Tikhonov regularization), this implies that the Hessian of $f$ in~\eqref{eq:problem},
\begin{equation}
\nabla^2 f(\bx) = \frac{1}{n}\sum_{i=1}^n \ell''(\ba_i^T\bx,b_i)(\ba_i\ba_i^T)+\lambda\bI,\label{eq:Hess_ERM}
\end{equation}
tends to be sparse. Based on this observation, we propose a strategy that aims to approximate 
$\nabla^2 f(\bx)$ {by} several smaller Hessian matrices and  update them efficiently. 
For sparse data, collecting curvature information via smaller dense Hessians can be more effective than {directly manipulating the} high-dimensional sparse Hessian~\cite{Nocedal_06}. As a result, the algorithm converges faster in practice (see Figure~\ref{fig:lowdim_Hess}). 

Before describing our strategy, we first introduce some notations. 
We partition $[n]$ into $K$ groups, 
and denote the set of partitions as $\calP\defeq\{\calP_1,\ldots,\calP_K\}$. For any $i\in[K]$, we define $\calS_i\defeq \cup_{j\in\calP_i}\supp(\ba_j)$, where $\supp(\ba_j)$ denotes the support of the vector $\ba_j$. 
We define $d_i\defeq \abs{\calS_i}$ and denote the elements in $\calS_i$ as $\{s_{i,1},\ldots,s_{i,d_i}\}$. 
We also define $F_i = \sum_{j\in\calP_i}f_j$ so that $f = \frac{1}{n}\sum_{i=1}^K F_i$.  
We define a {\em projection matrix} 
$\bU_i\in\bbR^{d_i\times d}$ such that for any $p\in[d_i]$ and $q\in[d]$, $u_{pq} = 1$ if $q = s_{i,p}$ and $0$ otherwise. Accordingly, for any $l\in\calP_i$, define a function $\phi_{i,l}:\bbR^{d_i}\to\bbR$ such that $f_l\defeq \phi_{i,l}\circ\bU_i$. Note that $\phi_{i,l}$ is uniquely defined by the definition $\bU_i$. Also define $\phi_i\defeq \sum_{l\in\calP_i}\phi_{i,l}$. 
 Therefore, 
\begin{align}
\nabla^2 f(\bx) 
= \frac{1}{n}\sum_{i=1}^K\bU_i^T\nabla^2 \phi_i(\bU_i\bx)\bU_i, \;\forall\,\bx\in\bbR^d.  \label{eq:sum_lowdim_Hess}
\end{align}
We now describe our strategy. 
In Algorithm~\ref{algo:SLBFGS-nonuniform}, for  any $i\in[K]$ and any $j\in\{r-M'+1,\ldots,r\}$, define correction pairs $\bs_{j,i}\!\!\defeq\!\! \bU_i(\barbx_{j}-\barbx_{j-1})$ and 
$\by_{j,i}\!\! \defeq\!\! \sum_{l\in\calT_{r,i}}\!\!\nabla^2 \phi_{i,l} (\bU_i\barbx_j)\bs_{j,i}$, 
where $\calT_{r,i}$ with size $b_\rmH/K$ is uniformly sampled from $\calP_i$. 
 Accordingly, define $\bS_{r,i}\defeq [\bs_{r-M'+1,i},\ldots,\bs_{r,i}]$, $\bY_{r,i}\defeq [\by_{r-M'+1,i},\ldots,\by_{r,i}]$. 
Instead of storing $\calH_r$, we only store matrices $\{(\bS_{r,i},\bY_{r,i})\}_{i=1}^K$. 
To reconstruct approximation $\bB_{r,i}$ for each $\nabla^2 \phi_i$ at $\bU_i\barbx_r$, as usual, we apply $M'$ BFGS updates~\eqref{eq:update_B} (using the correction pairs stored in $\bS_{r,i}$ and $\bY_{r,i}$) to $\bB_{r,i}^{(0)}\defeq \delta_{r,i}\bI$, where $\delta_{r,i}\!\defeq\! \normt{\by_{r,i}}^2/{\bs_{r,i}^T\by_{r,i}}$. This procedure can be implemented {\em efficiently} via the method of {\em compact representation}~\cite{Nocedal_06}, i.e., 
\begin{align}
\bB_{r,i} &= \delta_{r,i}\bI - \bW_{r,i}\bM_{r,i}^{-1}\bW_{r,i}^T,\label{eq:compact_rep}\\
\bM_{r,i} &\defeq \begin{bmatrix}
\delta_{r,i}{\bS_{r,i}}^T\bS_{r,i} &\bL_{r,i}\\
{\bL_{r,i}}^T& -\bD_{r,i}
\end{bmatrix},
\end{align}
where $\bW_{r,i}\!\defeq\! [\delta_r^{(i)}{\bS_r^{(i)}},\bY_r^{(i)}]$ and $\bL_{r,i}$ and $\bD_{r,i}$ are the lower triangular matrix (excluding diagonal) and diagonal matrix of $\bS_{r,i}^T\bY_{r,i}$ respectively. Analogous to \eqref{eq:sum_lowdim_Hess}, we define the approximation of $\nabla^2 f$ at $\bU_i\barbx_r$, denoted as $\bB_r$, as 
\begin{equation}
 \bB_r \defeq \frac{1}{n}\sum_{i=1}^K \bU_i^T\bB_{r,i}\bU_i.\label{eq:def_B_r} 
\end{equation}
We remark that the strong convexity of each function $f_i$ (see Assumption~\ref{assump:sc_sm}), together with the full-row-rank property of $\bU_i$, ensures $\nabla^2 \phi_i\!\succ\! 0$ on $\bbR^{d_i}$. This implies {\em positive curvature} $\bs_{r,i}^T\by_{r,i}>0$ and hence the positive definiteness of $\bB_{r,i}$, for any $r\!\ge\! 0$ and $i\!\in\![K]$. As a result, $\bB_r\!\succ\! 0$ on $\bbR^d$. 
This suggests the usage of the {\em conjugate gradient} (CG) method to compute the search direction at time $(s,t)$, namely $\bp_{s,t}\defeq -\bB_r^{-1}\bv_{s,t}$, via solving the positive definite system $\bB_r\bp_{s,t}=-\bv_{s,t}$. In particular, for any $\bz\in\bbR^d$, $\bB_r\bz$ and $\bz^T\bB_r\bz$ in CG can be computed very efficiently using \eqref{eq:compact_rep} and \eqref{eq:def_B_r}. For example, 
\begin{align}
\bz^T\bB_r\bz=\frac{1}{n}\sum_{i=1}^K \delta_{r,i}\norm{\bz_i}^2 -(\bW_{r,i}^T\bz_i)^T
\bM_{r,i}^{-1}(\bW_{r,i}^T\bz_i),\label{eq:quad_form}\vspace{-.5cm}
\end{align} 
where $\bz_i\defeq \bU_i\bz$. We observe that the total computational cost in \eqref{eq:quad_form} is $O(M'({M'}^2+d'))$, where $d'\defeq \sum_{i=1}^K d_i$. For sparse data, usually $d'=O(d)$, so this cost is still linear in $d$. In addition, 
we can compute \eqref{eq:quad_form} {\em in  parallel} across $i\in[K]$. 
(Intuitively, this amounts to {\em collecting curvature} from each function $\phi_i$ in parallel.)
In this case, the computational time will be greatly reduced to $O(M'({M'}^2+\max_{i}d_i))$. Since typically $\max_{i}d_i \ll d$, 
the computational savings from {\em parallel curvature collection} can be significant.\footnote{The memory parameter $M$ (note that $M'\le M$) is usually set to a small constant, e.g., 5 or 10. Thus it has less effect on the computational complexity compared to $d$ or $\max_{i}d_i$.}

\begin{remark}\label{rmk:Hessian_sketch}
Note that if we interpret the matrices $\{\bU_i^T\}_{i\in[K]}$ as {\em sketching matrices}~\cite{Woodruff_14}, then the acceleration technique in Section~\ref{sec:low_dim_Hessian} can be regarded as a way of performing (approximate) {\em Hessian sketching}. However, most existing methods in the literature~\cite{Pilanci_16,Luo_16,Gower_16,Pilanci_17} either use random sketching matrices or the (deterministic) frequent directions approach~\cite{Ghash_16}, which is based on the singular value decomposition. A certain amount of information contained in the Hessian is lost or modified in these sketching processes. In contrast, by using the sparse binary matrices $\{\bU_i^T\}_{i\in[K]}$, our approach merely (deterministically) compresses the large sparse Hessian matrix into $K$ small dense Hessians, without changing any information contained therein. 
\end{remark}


\begin{remark}
In~\cite{Gower_16}, the authors proposed another acceleration strategy called the {\em block BFGS update}~\cite{Gower_14,Hennig_15}. We remark that this strategy can be straightforwardly combined with all the other acceleration strategies proposed above, and may result in further acceleration of the convergence of Algorithm~\ref{algo:SLBFGS-nonuniform}.  
\end{remark}


\section{Numerical Experiments}\label{sec:experiments}

\subsection{Experimental Setup}\label{sec:setup}
We consider two ERM problems, including logistic regression (with Tikhonov regularization) and ridge regression. For logistic regression, $b_i\in\{-1,1\}$ and 
\begin{equation}
f_i^{\log}(\bx) \defeq \log\left(1+e^{-b_i(\ba_i^T\bx)}\right)+\frac{\lambda}{2} \norm{\bx}^2, \;\forall\,i\in[n].  \label{eq:logit}
\end{equation}
For ridge regression, $b_i\in\bbR$ and
\begin{equation}
f_i^{\rid}(\bx) \defeq \left(\ba_i^T\bx-b_i\right)^2+\frac{\lambda}{2} \norm{\bx}^2, \;\forall\,i\in[n].  \label{eq:ridge}
\end{equation}
Accordingly, define $f^{\log}\!\!\defeq \!\frac{1}{n}\!\sum_{i=1}^n f_i^{\log}$ and $f^{\rid}\!\!\defeq\! \frac{1}{n}\!\sum_{i=1}^n f_i^{\rid}$. 
Simple calculations reveal that the smoothness parameters $L_i$ of $f_i^{\log}$ and $f_i^{\rid}$ 
are given by $\norm{\ba_i}^2\!\!\!/4+\!\lambda$ and $2\norm{\ba_i}^2\!+\!\lambda$, respectively. Define the data matrix $\bA\defeq [\ba_1,\ldots,\ba_n]$. 
The condition numbers $\kappa$ of $f^{\log}$ and $f^{\rid}$ 
are given by $(\frac{1}{4n}\sigma_{\max}^2(\bA)+\lambda)/\lambda$ and $({2}\sigma_{\max}^2(\bA)\!+\!n\lambda)/({2}\sigma_{\min}^2(\bA)+n\lambda)$ respectively, where $\sigma_{\max}(\bA)$ and $\sigma_{\min}(\bA)$ denote the largest and smallest singular values of $\bA$ respectively.  
In both \eqref{eq:logit} and \eqref{eq:ridge}, we choose $\lambda=1/n$, following the convention in the literature (e.g.,~\cite{Gower_16}). 

We tested logistic and ridge regression problems on {\tt rcv1.binary} and {\tt E2006-tfidf} datasets respectively~\cite{Chang_11}. (In the sequel we abbreviate them as {\tt rcv1} and {\tt E2006-tf}.) 
From \eqref{eq:Hess_ERM}, we defined a sparsity estimate of $\nabla^2 f$ at any $\bx\in\bbR^d$ as $\varrho\defeq \abs{\supp(\bA\bA^T+\bI)}/d^2$. 
The statistics of both datasets, including the (ambient) data dimension $d$, number of data samples $n$, sparsity parameter $\varrho$ and condition number $\kappa$ (of $f^{\log}$ or $f^{\rid}$ defined by the datasets), are summarized in Table~\ref{table:first_two}.\footnote{The data dimension $d$ for the original {\tt E2006-tf} dataset is 150360. Due to memory issues, we randomly subsampled its features so that $d = 15036$.} From it, we observe that both datasets are large-scale and sparse, but with different $d$-to-$n$ ratios and condition numbers. Through these differences, we are able to infer the reasons for the different performances of some acceleration strategies on different ERM problems (shown in Section~\ref{sec:perf_acc}).

For both datasets, the norms of all feature vectors $\{\ba_i\}_{i=1}^n$ have been normalized to unit. Since the smoothness parameters $L_i$ for both ERM problems are only  dependent on $\norm{\ba_i}$ and $\lambda$, we have $L_i=L_j$ for any $i,j\in[n]$. Therefore the nonuniform distribution $P$ in Section~\ref{sec:algo} becomes uniform, and the merit of nonuniform sampling of $\calB_{s,t}$ cannot be observed.
 
To estimate the global optimum $\bx^*$ as ground truth, we used batch L-BFGS-B algorithm~\cite{Zhu_97}.
We randomly initialized $\bx^0$ according to a scaled standard normal distribution. (The performance of our algorithms were observed to be insensitive to the initialization of $\bx^0$.) We used the number of data passes (i.e., number of data points accessed divided by $n$), rather than running time, to measure the convergence rates of all the algorithms under comparison. This has been a well-established convention in the literature on both stochastic first-order~\cite{Johnson_13,Defazio_14,Reza_15} and second-order methods~\cite{Byrd_16a,Moritz_16,Gower_16} to make convergence results agnostic to the actual implementation of the algorithms, e.g., programming languages.\footnote{Specifically, our algorithm (Algorithm~\ref{algo:SLBFGS-nonuniform}) was implemented in Matlab\textsuperscript\textregistered. However, some benchmarking algorithms (see Section~\ref{sec:comp_other_algo}) were implemented in other languages, e.g., the SVRG algorithm~\cite{Johnson_13} and the stochastic L-BFGS algorithm~\cite{Moritz_16} were implemented in C++ and Julia respectively. The differences in programming language typically have a significant impact on the actual running time of algorithms.} 

Finally we describe the parameter setting. We set the mini-batch size $b=\sqrt{n}$, Hessian update period $\Upsilon=10$ and the memory parameter $M=10$. We set $b_\rmH=b\Upsilon$ so that the computation for $\by_r$ can be amortized to each inner iteration. We set the number of inner iterations $m=n/b$, so that each outer iteration will access $2n$ data points. Lastly, we set $\eta=1\times10^{-2}$. From Figure~\ref{fig:etas}, we observe that when $\eta$ is too large, e.g., $\eta=0.1$, Algorithm~\ref{algo:SLBFGS-nonuniform} only converges sublinearly; whereas when $\eta$ is too small, e.g., $\eta=1\times10^{-3}$, Algorithm~\ref{algo:SLBFGS-nonuniform} converges linearly but slowly. This corresponds well to our theoretical analysis in Theorem~\ref{thm:main}, which indicates that when $\eta$ falls below a threshold, $\rho$ increases as $\eta$ decreases. For both ERM problems, we see that $\eta=1\times10^{-2}$ yields fast linear convergence. 



\begin{table}\centering
\caption{Statistics of {\tt rcv1} and {\tt E2006-tf} datasets.\label{table:first_two}}
\begin{tabular}{|c|c|c|c|c|c|}\hline
Datasets & $d$ &$n$ & $d/n$ & $\varrho$ & $\kappa$\\\hline
{\tt rcv1} & 47236 & 20242 & 2.33 & 0.0154 & 113.17\\\hline
{\tt E2006-tf} & 15036 &16087 & 0.93 & 0.0404 & 1.70 \\\hline
\end{tabular}
\end{table}

\subsection{Performance of Acceleration Strategies}\label{sec:perf_acc}

We first examine the performance of Algorithm~\ref{algo:SLBFGS-nonuniform} with different schemes of choosing $\bx^{s+1}$. We consider five schemes in total, including (a) uniform sampling (option I), (b) uniform averaging (option II), (c) geometric sampling (option III), (d) geometric averaging (option IV) and (e) last inner iterate (in Remark~\ref{rmk:exp_samp_ave}). For options III and IV, we set $\beta=1/2$. From Figure~\ref{fig:geom_outer_itr}, we observe that options III and IV perform as well as the ``last inner  iterate'' scheme, on both ERM problems, and outperform the schemes based on uniform sampling/averaging {significantly}. For all the {subsequent} experiments, we use option~IV to select $\bx^{s+1}$. 

We next compare the performance of Algorithm~\ref{algo:SLBFGS-nonuniform} with and without using the subsampled gradient stabilization strategy in Section~\ref{sec:subsamp_grad_outer}. As suggested by Remark~\ref{rmk:part_grad}, we chose $\tilb_s=\min\{\zeta\upsilon^s,n\}$, where $\zeta= n/\upsilon^q$, $\upsilon=3$ and $q=8$. That is, the number of component gradients in $\tilbg_s$ exponentially increases in the first $p=8$ outer iterations and then remains at $n$. From Figure~\ref{fig:part_grad}, we observe that this simple parameter selection method works well on both ERM problems, especially in the initial phase (when $s$ is small). 
In addition, we also observe when $s$ is large, both algorithms have almost the same (linear) convergence rates. This corroborates our analysis in Proposition~\ref{prop:subsamp_grad_outer}. 

We finally compare the performance of Algorithm~\ref{algo:SLBFGS-nonuniform} with and without using the low-dimensional approximate Hessian strategy in Section~\ref{sec:low_dim_Hessian}. We set the number of partitions $K=5$ and partition $[n]$ evenly and randomly. From Figure~\ref{fig:lowdim_Hess}, we observe that our strategy leads to improvements of convergence on both logistic and ridge regression problems, and the improvement on the latter is {\em very significant}.  It could be possible that the smaller condition number of the {\tt E2006-tf} dataset enables more effective curvature collection by low-dimensional Hessians. Nevertheless, for the {\tt rcv1} dataset, which has a large condition number, our strategy is still efficacious. 
Additionally, we observe that our strategy {\em preserves {the} linear convergence} of Algorithm~\ref{algo:SLBFGS-nonuniform} on both problems. (A theoretical analysis of this linear convergence is left to future work.) 
Figure~\ref{fig:lowdim_Hess} shows the performance of our strategy in a single-threaded mode; as discussed in Section~\ref{sec:low_dim_Hessian}, our strategy can be much more efficient under scenarios where parallel computational resources are available.

\subsection{Comparison to Other Algorithms}\label{sec:comp_other_algo}
We combined all of our acceleration strategies in Section~\ref{sec:acc_strategy} and compared the resulting algorithm with three benchmarking algorithms, including SVRG in \cite{Johnson_13} (with mini-batch sampling of $\calB_{s,t}$) and two state-of-the-art stochastic L-BFGS algorithms in \cite{Moritz_16} and \cite{Gower_16}. For the algorithm in \cite{Gower_16}, we focused on its variant~(b), since it consistently outperformed other variants in experiments. 
We tested all the three benchmarking algorithms on both ERM problems with different step sizes $\eta\in\{10^{-4},10^{-3},\ldots,1\}$ and selected the best $\eta$ for each algorithm. The outer iterate $\bx^{s+1}$ in all these algorithms were selected via ``the last inner iterate'' scheme. 
From Figure~\ref{fig:diff_algo}, we observe that on both ERM problems, our algorithm yields faster convergence compared to all the benchmarking algorithms. 
This is due to the incorporation of the acceleration strategies in Sections~\ref{sec:subsamp_grad_outer} and~\ref{sec:low_dim_Hessian}. 
The improvement of convergence is particularly significant on the ridge regression problem. This observation is consistent with our observations in Section~\ref{sec:perf_acc}. 
In addition, we indeed observe that with the aid of curvature information, all the stochastic L-BFGS methods outperform the stochastic first-order method SVRG.

\section{Future work and an open problem }\label{sec:future_work}  
We propose to pursue future work in the following two directions.
First, we aim to develop and analyze {\em proximal} and {\em momentum-based accelerated} stochastic L-BFGS algorithms, based on our coordinate transformation framework. The proximal variant enables our algorithm to be applied to composite nonsmooth (convex) objective function. 
The accelerated variant can potentially improve the linear convergence rate of our algorithm, and thus reduce the total computational complexity. 
Second, we aim to analyze the convergence of the strategy in Section~\ref{sec:low_dim_Hessian}. 
In particular, Figure~\ref{fig:lowdim_Hess} suggests that Algorithm~\ref{algo:SLBFGS-nonuniform} may still converge linearly under this strategy. 

Besides future work, there is also an open problem we hope to resolve. Although we have improved the linear convergence rate and computational complexity of Algorithm~\ref{algo:SLBFGS-nonuniform} as compared to those in \cite{Moritz_16} and \cite{Gower_16},  it seems our improved complexity in \eqref{eq:C} is still inferior to that of SVRG. In SVRG, the complexity is $O((n+\kappa)d\log(1/\epsilon))$~\cite{Xiao_14}, so our complexity~\eqref{eq:C} has an additional multiplicative factor $\kappa_\rmH$. This contradicts the experimental results in \cite{Moritz_16}, \cite{Gower_16} and Section~\ref{sec:comp_other_algo}, where stochastic L-BFGS-type algorithms have been repeatedly shown to outperform their first-order counterparts. {\em Therefore,  an interesting problem consists in  obtaining a (computational) complexity bound of the stochastic L-BFGS algorithm that is better (or at least as good as) that of SVRG. }
Indeed, a careful analysis   reveals that the additional $\kappa_\rmH$ arises from the uniform spectral bound of the metric matrices $\{\bH_r\}_{r\ge 0}$. This uniform bound is effectively a worst-case bound, and does not reflect the {\em local curvature information} contained in recent iterates at any time $(s,t)$. Since the judicious use of curvature information serves as a very important reason for the fast convergence of the stochastic quasi-Newton algorithms, 
such information should also be reflected in theoretical analysis as well (possibly in an adaptive spectral bound for $\{\bH_r\}_{r\ge 0}$). We believe an effective adaptive bound is critical for improving the   complexity result in~\eqref{eq:C}. 

Interestingly, an {\em incremental} quasi-Newton (IQN) method was proposed by Mokhtari~{\em et al}.~\cite{Mokh_17} recently. The proposed algorithm makes use of the aggregated optimization variables, as well as the aggregated gradients and approximate Hessians of all the component functions to reduce  the noise of gradient and Hessian approximations. As a result, it achieves the local superlinear convergence rate, but requires $\Theta(nd^2)$ storage space. (Note that most of the stochastic L-BFGS methods, such as~\cite{Byrd_16a,Moritz_16} and Algorithm~\ref{algo:SLBFGS-nonuniform}, only require $\Theta(d)$ memory.) 
The key idea in~\cite{Mokh_17} is to show the descent direction in the IQN algorithm asympotically converges to that of the Newton's method. If this condition holds, then the additional factor $\kappa_\rmH$ in the complexity estimate~\eqref{eq:C} may be removed. Therefore, how to design a stochastic quasi-Newton algorithm that satisfies this condition in the {\em momeory-limited} setting would also be an interesting problem to pursue in the future. 

\begin{figure}[t]\centering
\subfloat[logistic regression ({\tt rcv1})]{\includegraphics[width=.5\columnwidth,height=.4\columnwidth]{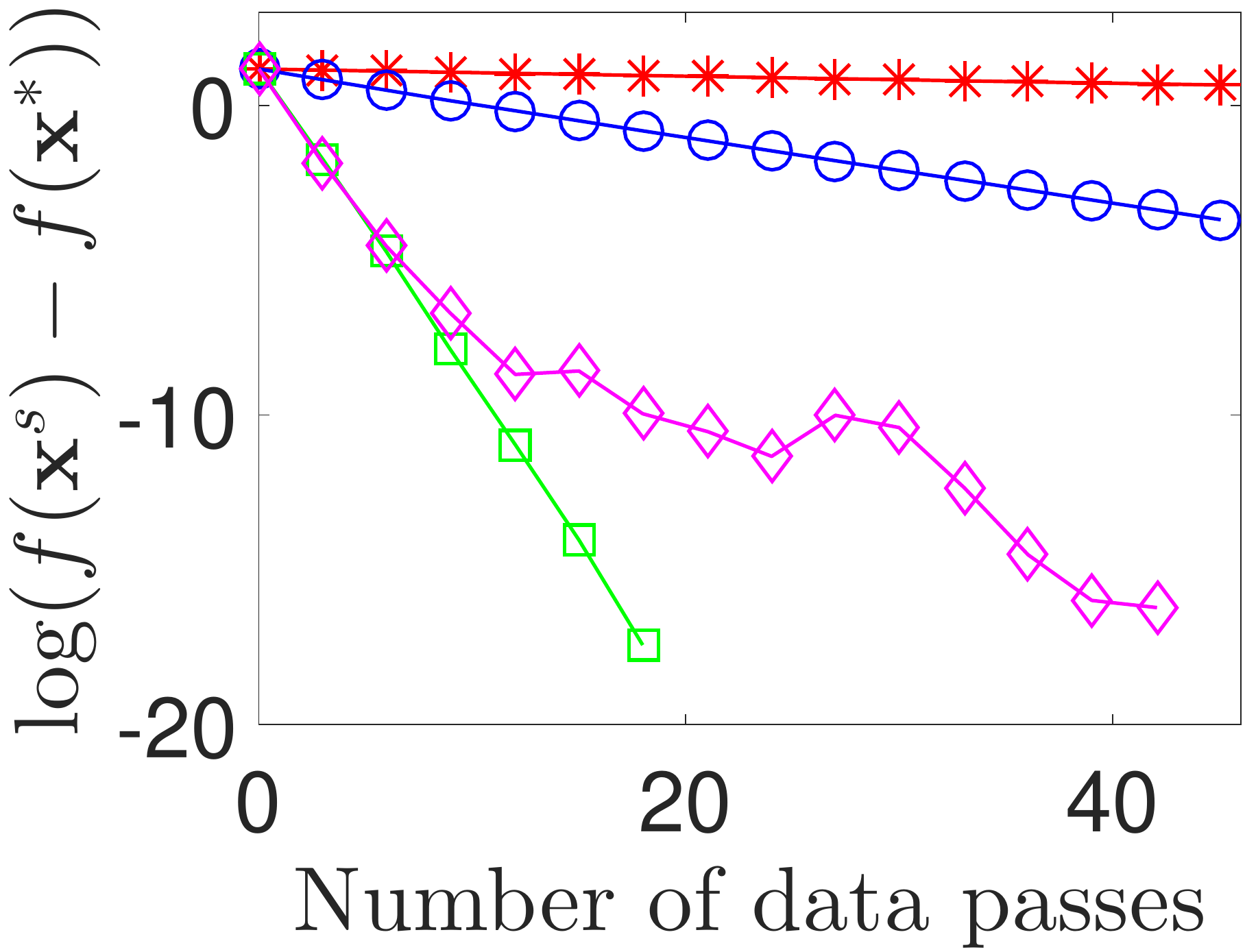}}\hfill
\subfloat[ridge regression ({\tt E2006-tf})]{\includegraphics[width=.47\columnwidth,height=.4\columnwidth]{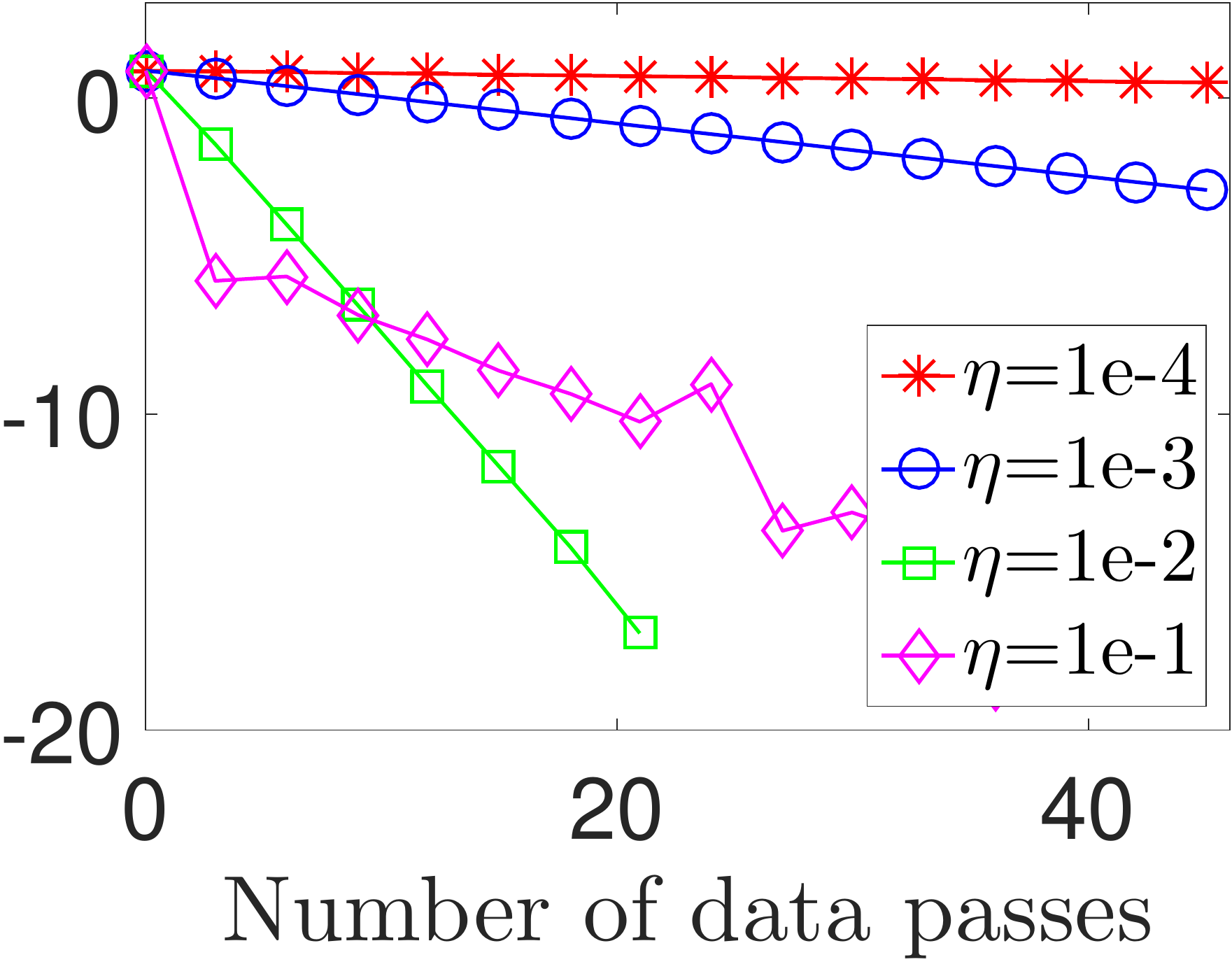}}
\caption{Log suboptimality versus number of passes through data of Algorithm~\ref{algo:SLBFGS-nonuniform} with different step sizes $\eta$. \label{fig:etas}}
\end{figure}

\begin{figure}[t]\centering
\subfloat[logistic regression ({\tt rcv1})]{\includegraphics[width=.5\columnwidth,height=.4\columnwidth]{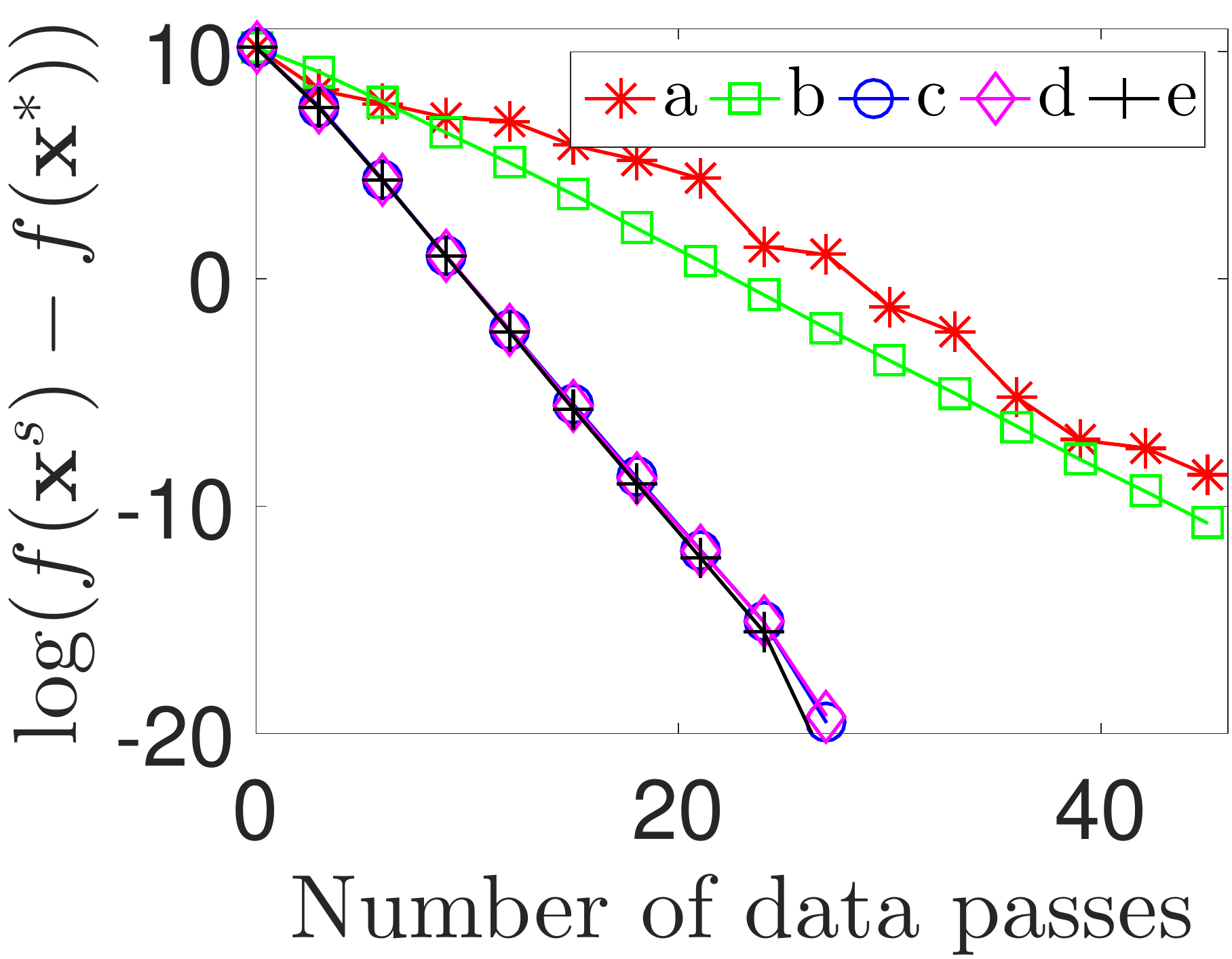}}\hfill
\subfloat[ridge regression ({\tt E2006-tf})]{\includegraphics[width=.47\columnwidth,height=.4\columnwidth]{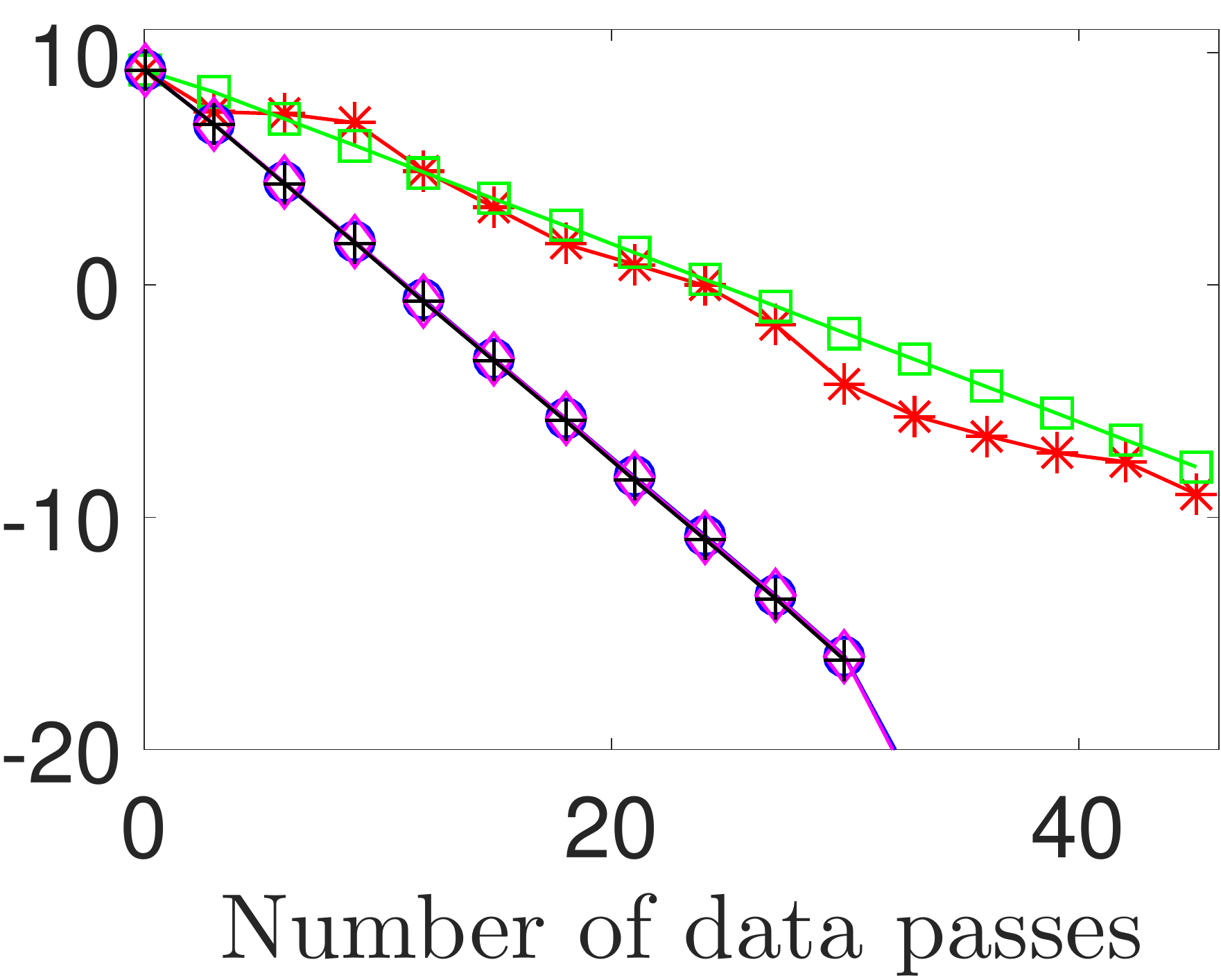}}
\caption{Comparison of Algorithm~\ref{algo:SLBFGS-nonuniform} with different selection schemes for $\bx^{s+1}$.\label{fig:geom_outer_itr} }
\end{figure}

\begin{figure}[t]\centering
\subfloat[logistic regression ({\tt rcv1})]{\includegraphics[width=.5\columnwidth,height=.4\columnwidth]{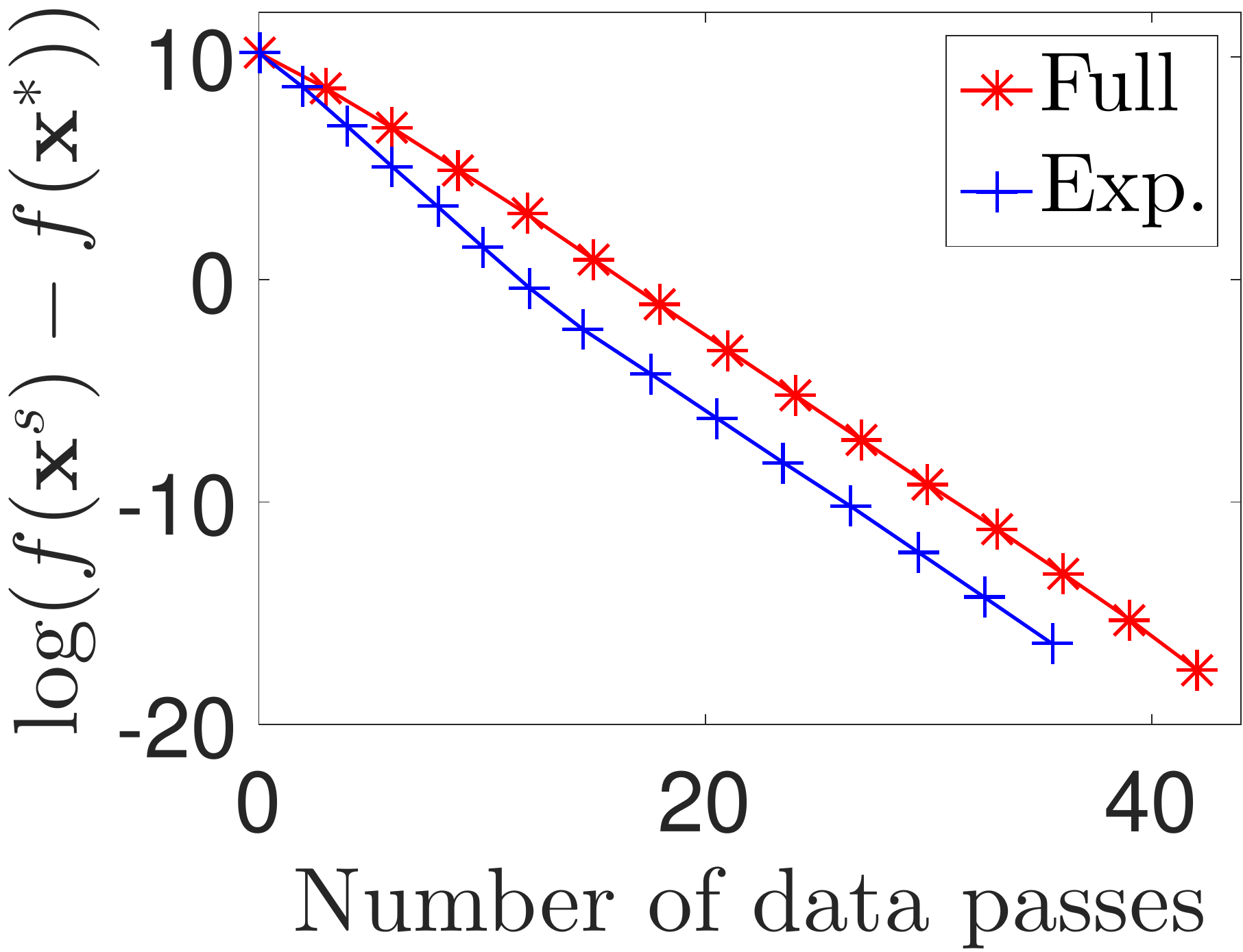}}\hfill
\subfloat[ridge regression ({\tt E2006-tf})]{\includegraphics[width=.47\columnwidth,height=.4\columnwidth]{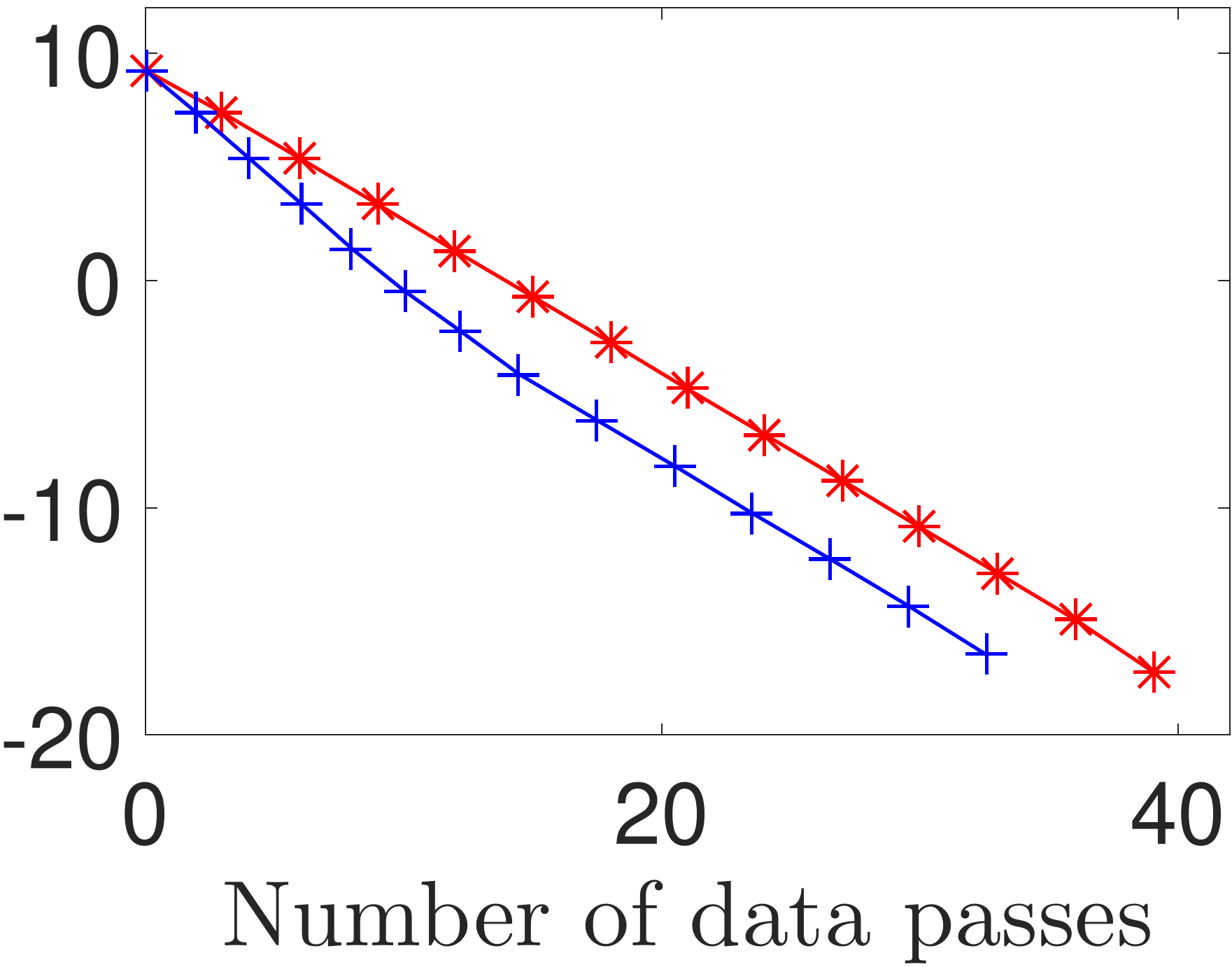}}
\caption{Comparison of Algorithm~\ref{algo:SLBFGS-nonuniform} without (Full) and with (Exp.) using the partial gradient statblization strategy in Section~\ref{sec:subsamp_grad_outer}.\label{fig:part_grad} }
\end{figure}

\begin{figure}[t]\centering
\subfloat[logistic regression ({\tt rcv1})]{\includegraphics[width=.5\columnwidth,height=.4\columnwidth]{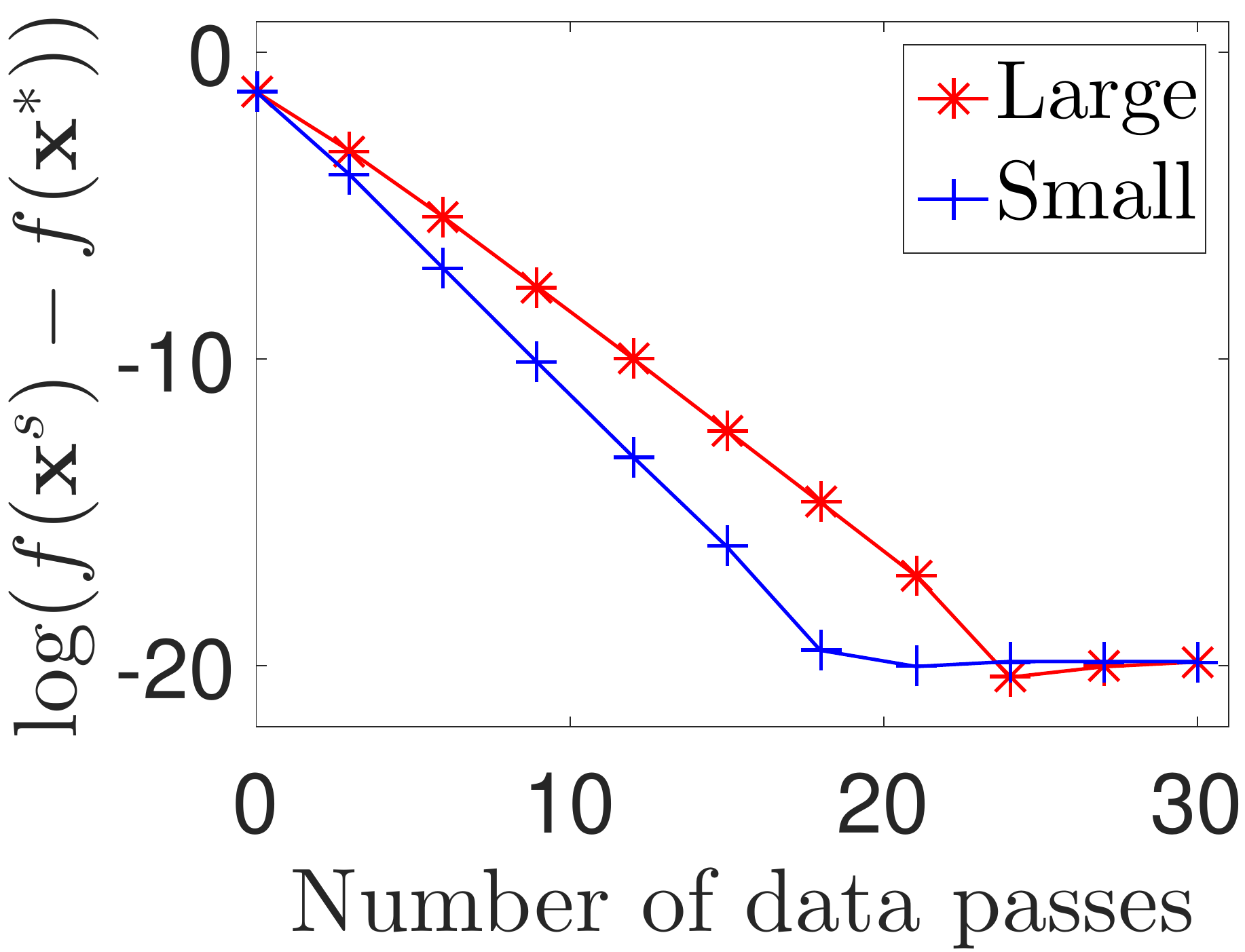}}\hfill
\subfloat[ridge regression ({\tt E2006-tf})]{\includegraphics[width=.47\columnwidth,height=.4\columnwidth]{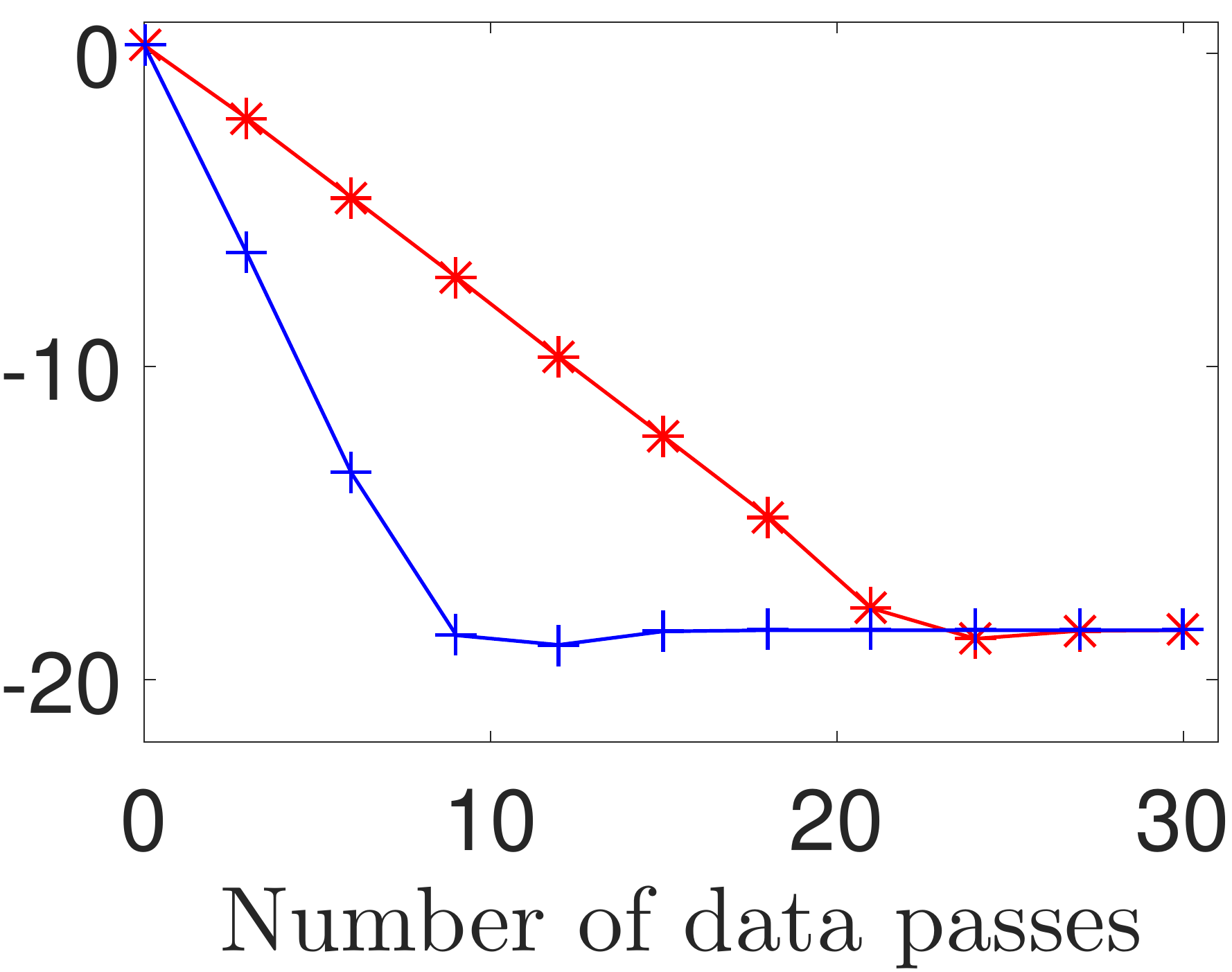}}
\caption{Comparison of Algorithm~\ref{algo:SLBFGS-nonuniform} without (Large) and with (Small) using the low-dimensional Hessian strategy in Section~\ref{sec:low_dim_Hessian}.\label{fig:lowdim_Hess} }
\end{figure}

\begin{figure}[t]\centering
\subfloat[logistic regression ({\tt rcv1})]{\includegraphics[width=.5\columnwidth,height=.4\columnwidth]{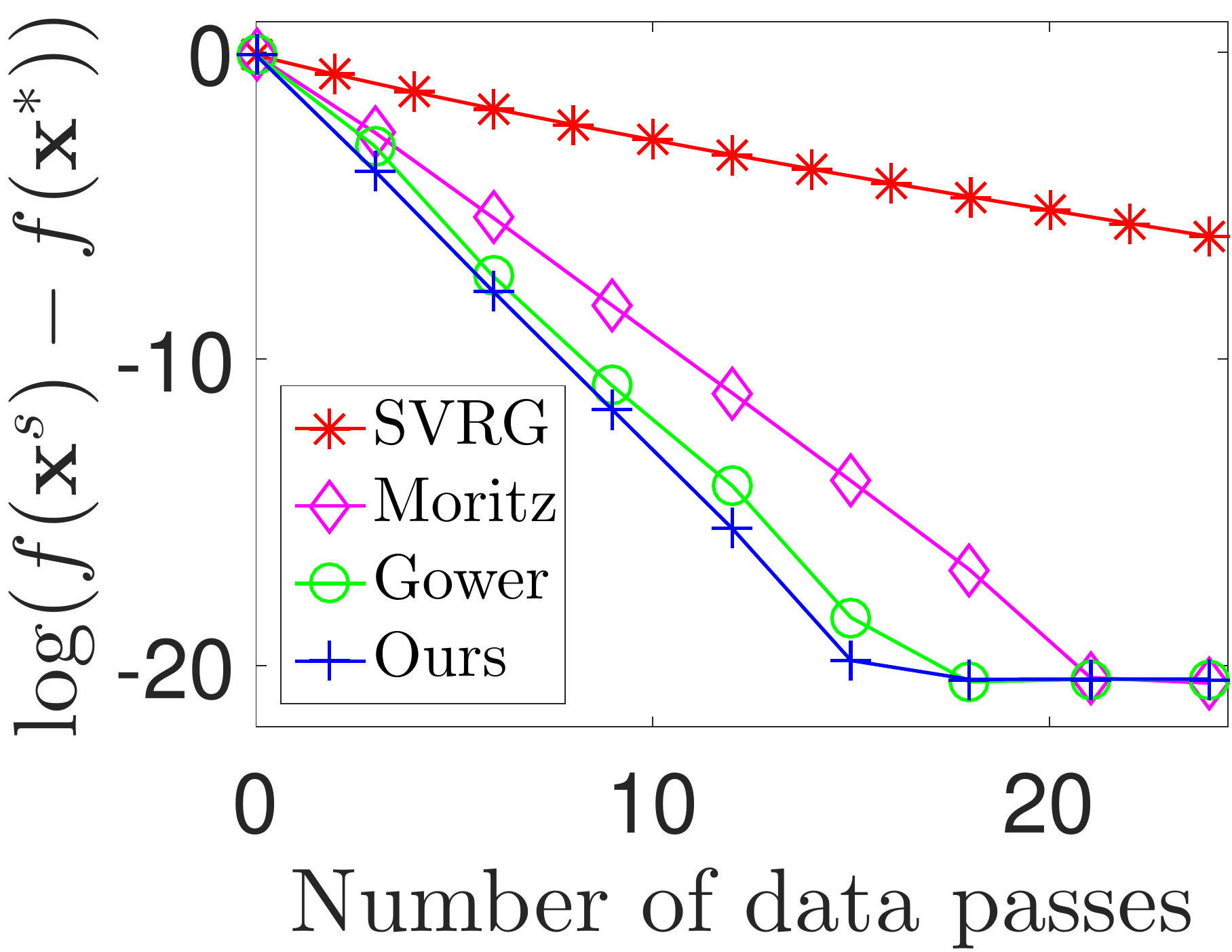}}\hfill
\subfloat[ridge regression ({\tt E2006-tf})]{\includegraphics[width=.47\columnwidth,height=.4\columnwidth]{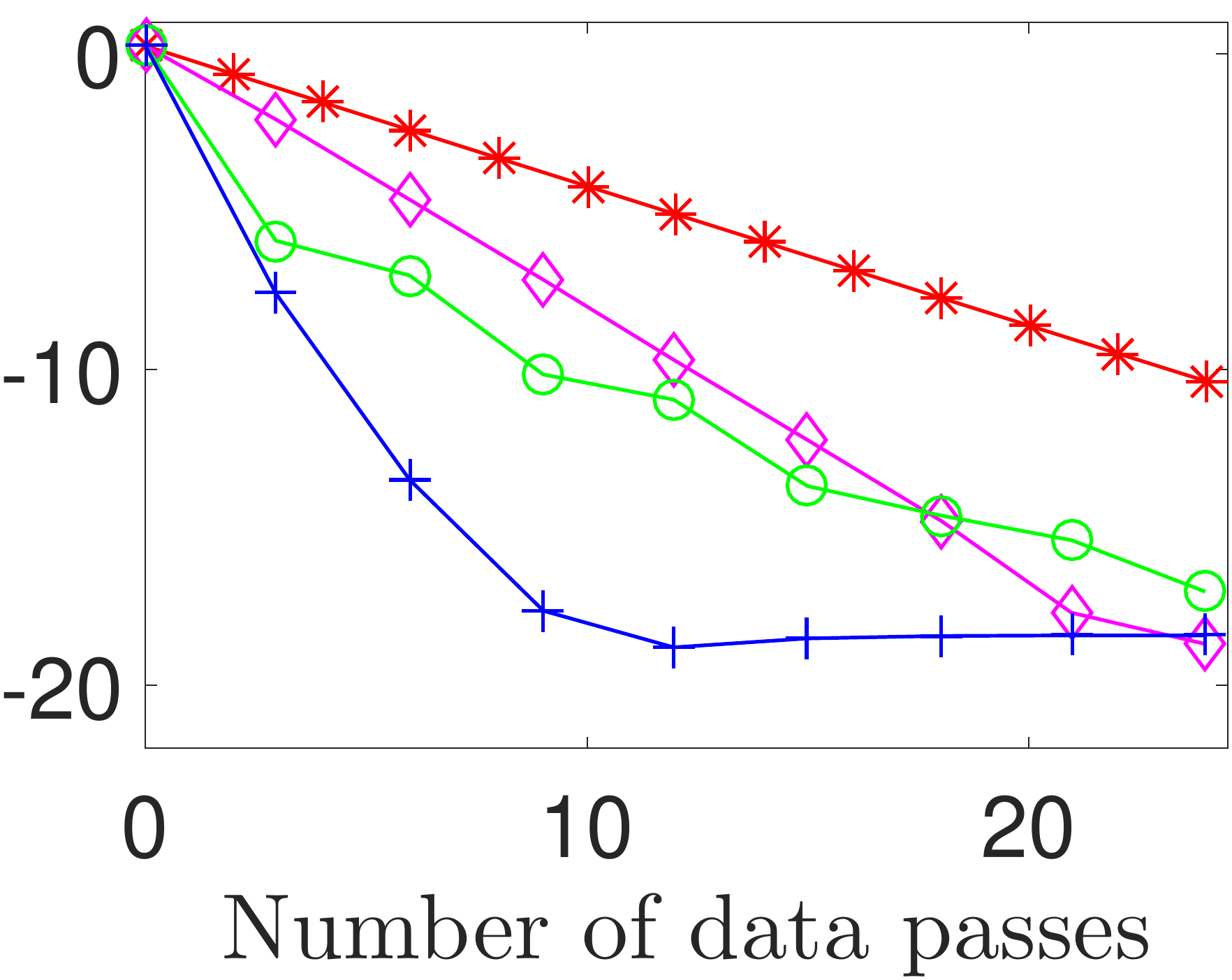}}
\caption{Comparison of our algorithm (Ours) with benchmarking algorithms (SVRG, Moritz and Gower) on the {\tt rcv1} and {\tt E2006-tf} datasets.}\label{fig:diff_algo} 
\end{figure}


\renewcommand{\thedefinition}{A-\arabic{definition}}
\renewcommand{\thelemma}{T-\arabic{lemma}}
\renewcommand{\thecorollary}{A-\arabic{corollary}}
\renewcommand{\thesection}{A-\arabic{section}}
\renewcommand{\theremark}{A-\arabic{remark}}
\renewcommand{\thefigure}{A-\arabic{figure}}
\renewcommand{\thetable}{A-\arabic{table}}
\renewcommand{\thealgorithm}{A-\arabic{algorithm}}

\setcounter{lemma}{0}
\setcounter{theorem}{0}
\setcounter{section}{0}
\setcounter{remark}{0}

\appendices

\section{Proof of Lemma~\ref{lem:bound_var}}\label{sec:proof_bound_var}

The proof of Lemma~\ref{lem:bound_var} is shown in \eqref{eq:bound_var_first} to \eqref{eq:bound_var_led} on page~\pageref{fig:equations}. 
\begin{figure*}[h!] \label{fig:equations}
\begin{align}
&\bbE_{\calB_{s,t}}\left[\norm{\bv_{s,t}-\nabla f(\bx_{s,t})}^2\vert\calF_{s,t}\right]\label{eq:bound_var_first}\\
=\;& \bbE_{\calB_{s,t}}\bigg[\Big\Vert\frac{1}{b}\sum_{j=1}^b\left(1/(np_{i_j})\left(\nabla f_{i_j}(\bx_{s,t})-\nabla f_{i_j}(\bx^s)\right)+\nabla f(\bx^s)-\nabla f(\bx_{s,t})\right)\Big\Vert^2\Big\vert\calF_{s,t}\bigg]\\
=\;&\frac{1}{b^2} \sum_{j=1}^b\bbE_{i_j}\left[\left.\norm{1/(np_{i_j})\left(\nabla f_{i_j}(\bx_{s,t})-\nabla f_{i_j}(\bx^s)\right)+\nabla f(\bx^s)-\nabla f(\bx_{s,t})}^2\right\vert\calF_{s,t}\right]\label{eq:bound_var_eqa}\\
\le\;&\frac{1}{b^2} \sum_{j=1}^b\bbE_{i_j}\left[\left.\norm{1/(np_{i_j})\left(\nabla f_{i_j}(\bx_{s,t})-\nabla f_{i_j}(\bx^s)\right)}^2\right\vert\calF_{s,t}\right]\label{eq:bound_var_leb}\\
=\;&\frac{1}{b^2} \sum_{j=1}^b\bbE_{i_j}\left[\left.\norm{1/(np_{i_j})\left(\nabla f_{i_j}(\bx_{s,t})-\nabla f_{i_j}(\bx^*)\right)+1/(np_{i_j})\left(\nabla f_{i_j}(\bx^*)-\nabla f_{i_j}(\bx^s)\right)}^2\right\vert\calF_{s,t}\right]\\
\le\;&\frac{2}{b^2} \sum_{j=1}^b\bbE_{i_j}\left[\left.\norm{1/(np_{i_j})\left(\nabla f_{i_j}(\bx_{s,t})-\nabla f_{i_j}(\bx^*)\right)}^2\right\vert\calF_{s,t}\right]+\bbE_{i_j}\left.\left[\norm{1/(np_{i_j})\left(\nabla f_{i_j}(\bx^s)-\nabla f_{i_j}(\bx^*)\right)}^2\right\vert\calF_{s,t}\right]\label{eq:bound_var_lec}\\
\le\;&\frac{2}{b^2} \sum_{j=1}^b 2\barL(f(\bx_{s,t})-f(\bx^*))+2\barL(f(\bx^{s})-f(\bx^*))
=\; \frac{4\barL}{b} \left(f(\bx_{s,t})-f(\bx^*)+f(\bx^s)-f(\bx^*)\right).\label{eq:bound_var_led}
\end{align}\hrule
\end{figure*}
In \eqref{eq:bound_var_eqa}, we use the independence of $i_j$ and $i_{j'}$ for $j\ne j'$ and \eqref{eq:expectation_nonuni} in Lemma~\ref{lem:bound_nonuniform}. 
In \eqref{eq:bound_var_leb}, we use \eqref{eq:expectation_nonuni} and the fact that $\bbE\big[\norm{\ba-\bbE[\ba]}^2\big\vert\calG\big]\le \bbE\big[\norm{\ba}^2\big\vert\calG\big]$ almost surely, for any $\sigma$-algebra $\calG$. In addition, the inequalities in~\eqref{eq:bound_var_lec} and \eqref{eq:bound_var_led}  follow from $\norm{\ba+\bb}^2\le 2\norm{\ba}^2+2\norm{\bb}^2$ 
and \eqref{eq:bound_norm_nonuni} respectively.
 
\section{Proof of Lemma~\ref{lem:spectral_bound}}\label{sec:proof_spec_bound}
Our proof is inspired by \cite{Gower_16}. For any $\bx\in\bbR^d$ and $\calT\subseteq [n]$ with cardinality $b_\rmH$, we have 
\begin{equation}
\barmu_{b_\rmH}\bI \preceq \nabla^2 f_{\calT}(\bx) \preceq \barL_{b_\rmH} \bI.
\end{equation}
Define $\bV_k \defeq \bI-{\by_k{\bs_k}^T}/({\by_k}^T\bs_k)$ and $\bQ_k \defeq {\bs_k{\bs_k}^T}/({\by_k}^T\bs_k)$, then \eqref{eq:update_H} becomes
\begin{equation}
\bH^{(k)}_r = \bV_k\bH^{(k-1)}_r\bV_k^T + \bQ_k.\label{eq:update_H_VQ}
\end{equation} 
Fix any $r\in\bbN$. 
Since $\by_k = \nabla^2 f_{\calT_r}(\barbx_r)\bs_k$, we have 
\begin{align*}
\bV_k &= \bI-\frac{\nabla^2 f_{\calT_r}(\barbx_r)\bs_k{\bs_k}^T}{{\bs_k}^T\nabla^2 f_{\calT_r}(\barbx_r)\bs_k}\\
&= \nabla^2 f_{\calT_r}(\barbx_r)^{1/2}\left(\bI - \frac{\tilbs_k{\tilbs_k}^T}{{\tilbs_k}^T\tilbs_k}\right)\nabla^2 f_{\calT_r}(\barbx_r)^{-1/2},
\end{align*}
where $\tilbs_k \defeq \nabla^2 f_{\calT_r}(\barbx_r)^{1/2}\bs_k$. 
Hence 
\begin{align*}
\norm{\bV_k} &\le \norm{\nabla^2 f_{\calT_r}(\barbx_r)^{1/2}}\norm{\bI - \frac{\tilbs_k{\tilbs_k}^T}{\norm{\tilbs_k}^2}}\norm{\nabla^2 f_{\calT_r}(\barbx_r)^{-1/2}}\\
&\le \barL_{b_\rmH}^{1/2} \barmu_{b_\rmH}^{-1/2} = \kappa_{b_\rmH}^{1/2}.
\end{align*}
Similarly,
$\norm{\bQ_k} = {\norm{\bs_k}^2}/({\bs_k}^T\nabla^2 f_{\calT_r}(\barbx_r)\bs_k) \le {1}/\barmu_{b_\rmH}$. 
By~\eqref{eq:update_H_VQ}, 
\begin{align}
\norm{\bH^{(k)}_r} &\le \norm{\bV_k}^2\norm{\bH^{(k-1)}_r} + \norm{\bQ_k} \nn\\ 
&\le \kappa_{b_\rmH}\norm{\bH^{(k-1)}_r}+\frac{1}{\barmu_{b_\rmH}}. \label{eq:bound_H_k}
\end{align}
We apply \eqref{eq:bound_H_k} repeatedly over $k\!=\!r-M'+1,\ldots,r$, then 
\begin{align*}
\norm{\bH_r}=\norm{\bH_r^{(r)}} &\le \kappa_{b_\rmH}^{M'}\norm{\bH_r^{(r-M')}} + \frac{1}{\barmu_{b_\rmH}}\sum_{i=0}^{M'-1}\kappa_{b_\rmH}^i\\
&\lea \frac{1}{\barmu_{b_\rmH}}\left(\kappa_{b_\rmH}^{M'}+\frac{\kappa_{b_\rmH}^{M'}-1}{\kappa_{b_\rmH}-1}\right)\\
&\le \frac{1}{\barmu_{b_\rmH}}\kappa_{b_\rmH}^{M'}\left(1+\frac{1}{\kappa_{b_\rmH}-1}\right)\\
&\leb \frac{\kappa_{b_\rmH}^{M+1}}{\barmu_{b_\rmH}(\kappa_{b_\rmH}-1)},
\end{align*}
where (a) follows from the definition of $\bH_r^{(r-M')}$ in \eqref{eq:def_H0} and 
\begin{align}
\frac{{\bs_r}^T\by_r}{{\by_r}^T\by_r} &= \frac{{\tilbs_r}^T\tilbs_r}{{\tilbs_r}^T\nabla^2 f_{\calT_r}(\barbx_r)\tilbs_r}\le \frac{1}{\barmu_{b_\rmH}},
\end{align}
and (b) follows from $\kappa_{b_\rmH}\ge 1$ and $M'\le M$. 
To show $\gamma = {1}/{(M+1)\barL_{b_\rmH}}$, it suffices to show $\norm{\bB_r}\le (M+1)\barL_{b_\rmH}$.
We derive this bound using \eqref{eq:update_B}. Since 
\begin{align}
&\hspace{-.1cm}\quad \,\norm{\bB_r^{(k-1)}-\frac{\bB_r^{(k-1)}\bs_k{\bs_k}^T\bB_r^{(k-1)}}{{\bs_k}^T\bB_r^{(k-1)}\bs_k}} \nn\\
&\hspace{-.1cm}\le \norm{\left(\bB_r^{(k-1)}\right)^{1/2}}^2\norm{\bI-\frac{\hatbs_k\hatbs_k^T}{\norm{\hatbs_k}^2}} = \norm{\bB_r^{(k-1)}},\nn
\end{align}
and
\begin{align}
&\hspace{-.1cm}\quad\,\norm{\frac{\by_k\by_k^T}{\bs_k^T\by_k}} = \frac{\norm{\by_k}^2}{\bs_k^T\by_k} = \frac{\tilbs_k^T\nabla^2 f_{\calT_r}(\barbx_r)\tilbs_k}{\tilbs_k^T\tilbs_k} \le \barL_{b_\rmH},\label{eq:bound_B0}
\end{align}
we have from \eqref{eq:update_B} that 
\begin{equation}
\norm{\bB_r^{(k)}} \le \norm{\bB_r^{(k-1)}} + \norm{\frac{\by_k\by_k^T}{\bs_k^T\by_k}}\le \norm{\bB_r^{(k-1)}} + \barL_{b_\rmH}.\nn 
\end{equation}
Therefore, 
\begin{align}
\norm{\bB_r} = \norm{\bB_r^{(r)}} \le \norm{\bB_r^{(r-M')}} + M'\barL_{b_\rmH} \lea (M+1)\barL_{b_\rmH},\nn
\end{align}
where (a) follows from \eqref{eq:bound_B0}. 

\section{Proof of Proposition~\ref{prop:exp_samp_ave} }\label{sec:proof_exp_ave}
Our proof is modified from that of Theorem~\ref{thm:main}. Specifically, our proof leverages a refined telescoping of \eqref{eq:succ_dist_bound}. (The steps up to \eqref{eq:succ_dist_bound} are unchanged.)
For each $t\in[m]$, we multiply both sides of \eqref{eq:succ_dist_bound} by $(1-\eta\gamma\barmu)^{m-t}$ and obtain
\begin{align}
&\hspace{-.2cm}(1-\eta\gamma\barmu)^{m-t}\left(\left.\bbE_{\calB_{s,t}}\left[\norm{\tilbx_{s,t,r}-\tilbx_r^*}^2\right\vert\calF_{s,t}\right]\right.\nn\\
&\hspace{-.25cm}\left.\hspace{3cm}+2\eta\bbE_{\calB_{s,t}}\left.\left[\tilf_r (\tilbx_{s,t,r})-\tilf_r (\tilbx_{r}^*)\right\vert\calF_{s,t}\right]\right) \nn\\
&\hspace{-.25cm}\le(1-\eta\gamma\barmu)^{m-t}\left((1-\eta\gamma\barmu)\norm{\tilbx_{s,t-1,r} - \tilbx_r^*}^2\right. \nn\\
&\hspace{-.25cm}+\!\frac{8}{b}\Gamma\barL\eta^2(\tilf_r(\tilbx_{s,t-1,r})\!-\!\tilf_r(\tilbx_r^*)\!+\!\tilf_{r'}(\tilbx^{s,{r'}})\!-\!\tilf_{r'}(\tilbx_{r'}^*))\!\Big).  \label{eq:succ_dist_bound2}
\end{align}
Telescope \eqref{eq:succ_dist_bound2} over $t=1,\ldots,m$ and we have
\begin{align}
&2c\eta\left(1-\frac{4}{b}\frac{\Gamma\barL\eta}{1-\eta\gamma\barmu}\right)\nn\\
&\hspace{.5cm}\cdot\frac{1}{c}\sum_{t=1}^m(1-\eta\gamma\barmu)^{m-t}\bbE_{\calB_{s,(t]}}\left.\left[\tilf_r(\tilbx_{s,t,r})-\tilf_r(\tilbx^*_{r})\right\vert\calF_{s}\right]\nn\\
&\le(1-\eta\gamma\barmu)^{m}\norm{\tilbx^{s,r'}-\bx_{r'}^*}^2\nn\\
&\;+\frac{8}{b}\Gamma\barL\eta^2\left((1-\eta\gamma\barmu)^{m}+c\right)\left(\tilf_{r'}(\tilbx^{s,r'})-\tilf_{r'}(\tilbx_{r'}^*)\right).\label{eq:telescoped_uniform2}
\end{align}
Now we consider using option IV to choose $\bx^{s+1}$. 
Since $0<\beta\le 1-\eta\gamma\barmu$, using \eqref{eq:tilf} and Jensen's inequality, 
 we have
\begin{align}
&\frac{1}{c}\sum_{t=1}^m(1-\eta\gamma\barmu)^{m-t}\bbE_{\calB_{s,(t]}}\left.\left[\tilf_r(\tilbx_{s,t,r})-\tilf_r(\tilbx^*_{r})\right\vert\calF_{s}\right]\nn\\\
&\hspace{1.7cm}\ge \bbE_{\calB_{s,(m]}}\left.\left[ \tilf_{r''}(\tilbx^{s+1,r''})-\tilf(\tilbx^*_{r''})\right\vert\calF_s\right]. \label{eq:exp_Jensen_nonuni}
\end{align}
Alternatively, if $\bx^{s+1}$ is determined using option III, the definition of distribution $Q$ still yields \eqref{eq:exp_Jensen_nonuni}. 
Finally, using \eqref{eq:bound_sc} in Lemma~\ref{lem:sc_sm} to bound $\normt{\tilbx^{s,r'}-\tilbx_{r'}^*}^2$ in \eqref{eq:telescoped_uniform2}, we have
\begin{align}
&2c'\eta\left(1-\frac{4}{b}\frac{\Gamma\barL\eta}{1-\eta\gamma\barmu}\right)\bbE_{\calB_{s,(m]}}\left.\left[ \tilf_{r''}(\tilbx^{s+1,r''})-\tilf(\tilbx^*_{r''})\right\vert\calF_s\right]\nn\\
&\hspace{0cm}\le\left(\frac{8}{b}\Gamma\barL\eta^2(1+c')+\frac{2}{\gamma\barmu}\right)\left(\tilf_{r'}(\tilbx^{s,r'})-\tilf_{r'}(\tilbx_{r'}^*)\right).
\end{align}
Taking expectation on both sides and using \eqref{eq:tilf}, we arrive at~\eqref{eq:faster_linear_conv}.

\section{Proof of Proposition~\ref{prop:subsamp_grad_outer} }\label{sec:proof_subsamp_grad_outer}

The subsampled gradient strategy essentially introduces errors in $\{\bg_s\}_{s\ge 0}$. Therefore, we explicit model this error by $\be_s \defeq \tilbg_s - \bg_s$. 
We first bound the second moment of $\be_s$. 
Since $\bg_s = \nabla f(\bx^s)$, by Lemma~\ref{lem:uniform_samp_worep} and $\tilb_s \ge nS^2\alpha_s/(S^2\alpha_s +(n-1)\xi^2\barrho^{2s})$, we have for any $s\in(S]$,
\begin{equation}
\bbE_{\tilcalB_s}\left[\norm{\be_s}^2\right]\le \frac{n-\tilb_s}{\tilb_s(n-1)}\alpha_s\le \frac{\xi^2}{S^2}\barrho^{2s}. \label{eq:error_2nd_mom}
\end{equation}
As a result, 
\begin{equation}
\bbE_{\tilcalB_s}\left[\norm{\be_s}\right] \le \sqrt{\bbE_{\tilcalB_s}\left[\norm{\be_s}^2\right]} \le {\frac{\xi}{S}}\barrho^{s}. \label{eq:error_1st_mom}
\end{equation}
The introduction of random sets $\{\tilcalB_{s}\}_{s\ge 0}$ requires us to redefine the filtration $\{\calF_{s,t}\}_{s\ge 0, t\in(m-1]}$ as
\begin{align}
&\calF_{s,t}\defeq \sigma\Big(\{\tau_j\}_{j=0}^{s-1}\cup\{\tilcalB_{j}\}_{j\in (s]}\cup\{\calB_{i,j}\}_{i\in (s-1],j\in(m-1]}\nn\\
&\hspace{3cm}\cup\{\calB_{s,j}\}_{j=0}^{t-1}\cup\{\calT_j\}_{j=0}^{\floor{(sm+t)/L}}\Big).
\end{align}
As usual, we define $\calF_s\defeq \calF_{s,0}$. In the sequel, we follow the naming convention in coordinate transformation framework as in Section~\ref{sec:def}. In particular, we define $\tilbe_{s,r'} \defeq \bH_{r'}^{1/2}\be_{s}$. 
 Now, define $\tilbv'_{s,t,r} \defeq \tilbv_{s,t,r}+\tilbe_{s,r'}$, then from \eqref{eq:exp_nabla_tilf} and \eqref{eq:var_nabla_tilf}, we have 
\begin{align}
&\bbE_{\calB_{s,t}}\left.\left[\tilbv'_{s,t,r}\right\vert\calF_{s,t}\right] = \nabla \tilf_r (\tilbx_{s,t,r}) + \tilbe_{s,r'},\label{eq:exp_nabla_tilf2}\\
&\bbE_{\calB_{s,t}}\left.\left[\norm{\tilbv'_{s,t,r}-\nabla \tilf_r(\tilbx_{s,t,r})}^2\right\vert\calF_{s,t}\right]\le \frac{4\Gamma\barL}{b}(\tilf_r(\tilbx_{s,t,r})\nn\\
&\hspace{1cm}-\tilf_r(\tilbx_r^*)+\tilf_{r'}(\tilbx^{s,{r'}})-\tilf_{r'}(\tilbx_{r'}^*)) + \norm{\tilbe_{s,r'}}^2.\label{eq:var_nabla_tilf2}
\end{align} 
Define $\tbdelta'_{s,t,r} \defeq \tilbv'_{s,t,r}-\nabla \tilf_r (\tilbx_{s,t,r})$. We can derive an inequality similar to \eqref{eq:succ_bound_sqEuc} in the proof of Theorem~\ref{thm:main}, i.e., 
\begin{align*}
&\norm{\tilbx_{s,t+1,r}-\tilbx_r^*}^2 \!\le\! (1-\eta\gamma\barmu)\norm{\tilbx_{s,t,r} \!-\! \tilbx_r^*}^2 \!-\! 2\eta(\tilf_r (\tilbx_{s,t+1,r})\nn\\
&\hspace{1cm}-\tilf_r (\tilbx_{r}^*))+ {2\eta^2} \normt{\tbdelta'_{s,t,r}}^2 -2\eta\lrangle{\tbdelta'_{s,t,r}}{\tilbx_{s,t,r}-\tilbx_r^*},
\end{align*}
Taking expectation w.r.t. $\calB_{s,t}$ and using \eqref{eq:exp_nabla_tilf2} and \eqref{eq:var_nabla_tilf2}, we have
\begin{align}
&\left.\bbE_{\calB_{s,t}}\!\!\left[\norm{\tilbx_{s,t+1,r}-\tilbx_r^*}^2+2\eta(\tilf_r (\tilbx_{s,t+1,r})-\tilf_r (\tilbx_{r}^*))\right\vert\calF_{s,t}\right]\nn\\
&\le(1-\eta\gamma\barmu)\norm{\tilbx_{s,t,r} - \tilbx_r^*}^2 +\frac{8}{b}\Gamma\barL\eta^2(\tilf_r(\tilbx_{s,t,r})-\tilf_r(\tilbx_r^*)\nn\\
&\!+\!\tilf_{r'}(\tilbx^{s,{r'}})\!-\!\tilf_{r'}(\tilbx_{r'}^*)) \!+\! 2\eta^2\norm{\tilbe_{s,r'}}^2 \!-\!2\eta\lrangle{\tilbe_{s,r'}}{\tilbx_{s,t,r}\!-\!\tilbx_r^*}.  \label{eq:succ_dist_bound3}
\end{align}
Using Cauchy-Schwartz inequality, we have 
\begin{align}
&\hspace{-.35cm}-2\eta\lrangle{\tilbe_{s,r'}}{\tilbx_{s,t,r}-\tilbx_r^*}\le 2\eta\norm{\tilbe_{s,r'}}\norm{\tilbx_{s,t,r}-\tilbx_r^*}\nn\\
&\hspace{-.35cm}\le 2\eta\norm{\tilbe_{s,r'}}\normt{\bH_r^{-1/2}}\norm{\bx_{s,t}-\bx^*}\le 2\eta B\gamma^{-1/2} \norm{\tilbe_{s,r'}}. \label{eq:CS_bound}
\end{align} 
Substituting \eqref{eq:CS_bound} into \eqref{eq:succ_dist_bound3}, and using the telescoping techniques in Appendix~\ref{sec:proof_exp_ave}, we have
\begin{align}
&2c\eta\left(1-\frac{4}{b}\frac{\Gamma\barL\eta}{1-\eta\gamma\barmu}\right)\frac{1}{c}\sum_{t=1}^m(1-\eta\gamma\barmu)^{m-t}\nn\\
&\cdot\bbE_{\calB_{s,(t]}}\!\!\left.\left[\tilf_r(\tilbx_{s,t,r})-\tilf_r(\tilbx^*_{r})\right\vert\calF_{s}\right]\!\le\!(1-\eta\gamma\barmu)^{m}\normt{\tilbx^{s,r'}\!-\!\bx_{r'}^*}^2\nn\\
&\hspace{.6cm}+\frac{8}{b}\Gamma\barL\eta^2\left((1-\eta\gamma\barmu)^{m}+c\right)\left(\tilf_{r'}(\tilbx^{s,r'})-\tilf_{r'}(\tilbx_{r'}^*)\right)\nn\\
&\hspace{.6cm} + 2\eta\left((1-\eta\gamma\barmu)^{m}+c\right)\left(B\gamma^{-1/2} \norm{\tilbe_{s,r'}}+\eta\norm{\tilbe_{s,r'}}^2\right).\nn
\end{align}
With either option III or IV and \eqref{eq:bound_sc} in Lemma~\ref{lem:sc_sm}, we have
\begin{align}
&2c'\eta\left(1-\frac{4}{b}\frac{\Gamma\barL\eta}{1-\eta\gamma\barmu}\right)\bbE_{\calB_{s,(m]}}\left.\left[ \tilf_{r''}(\tilbx^{s+1,r''})-\tilf(\tilbx^*_{r''})\right\vert\calF_s\right]\nn\\
&\le\left(\frac{8}{b}\Gamma\barL\eta^2(1+c')+\frac{2}{\gamma\barmu}\right)\left(\tilf_{r'}(\tilbx^{s,r'})-\tilf_{r'}(\tilbx_{r'}^*)\right) \nn\\
&\hspace{1cm}+ 2\eta\left(1+c'\right)\left(B\gamma^{-1/2} \norm{\tilbe_{s,r'}}+\eta\norm{\tilbe_{s,r'}}^2\right).
\end{align}
Taking expectation on both sides and using \eqref{eq:tilf}, we have
\begin{align}
&\bbE\left[f(\bx^{s+1})-f(\bx^*)\right] \le \barrho\bbE\left[f(\bx^s)-f(\bx^*)\right]+\left(1+{1}/{c'}\right) \nn\\
&\cdot\frac{b}{b-{4\Gamma\barL\eta}/{(1-\eta\gamma\barmu)}}\left(B\gamma^{-1/2} \bbE[\norm{\tilbe_{s,r'}}]+\eta\bbE[\norm{\tilbe_{s,r'}}^2]\right). \label{eq:recur_err}
\end{align}
From \eqref{eq:error_2nd_mom} and \eqref{eq:error_1st_mom}, we have
\begin{align}
\bbE\left[\norm{\tilbe_{s,r'}}\right] &\le \bbE\left[\normt{\bH^{1/2}_{r'}}\norm{\be_{s}}\right] \le \frac{\Gamma^{1/2}\xi}{S}\barrho^{s},\label{eq:1st_mom_err_r}\\
\bbE\left[\norm{\tilbe_{s,r'}}^2\right] &\le \bbE\left[\normt{\bH^{1/2}_{r'}}^2\norm{\be_{s}}^2\right] \le {\frac{\Gamma \xi^2}{S^2}}\barrho^{2s}.\label{eq:2nd_mom_err_r}
\end{align}
Substituting \eqref{eq:1st_mom_err_r} and \eqref{eq:2nd_mom_err_r} into \eqref{eq:recur_err}, we have
\begin{align}
&\hspace{-.2cm}\bbE\left[f(\bx^{s+1})-f(\bx^*)\right] 
\le \barrho\bbE\left[f(\bx^s)-f(\bx^*)\right]\nn\\
&\hspace{-.3cm} +  \left(1\!+\!\frac{1}{c'}\right)\frac{b}{b-{4\Gamma\barL\eta}/{(1-\eta\gamma\barmu)}}\left(\kappa_\rmH^{1/2}B+{\eta\Gamma \xi}\right)\frac{\xi}{S}\barrho^{s}.\label{eq:recur_err2}
\end{align}
Applying \eqref{eq:recur_err2} recursively and we reach \eqref{eq:lin_conv_subsamp_grad}.

\section{Technical Lemmas}\label{app:tech_lemma}
Lemmas \ref{lem:sc_sm}. \ref{lem:bound_nonuniform}, and \ref{lem:uniform_samp_worep} can be found in \cite[Chapter~9]{Boyd_04}, \cite{Xiao_14}, and \cite{Nitanda_14} respectively.
\begin{lemma}\label{lem:sc_sm}  
    If a function $f:\bbR^d\to\bbR$ is $\mu$-strongly convex and $L$-smooth, then for any $\bx\in\bbR^d$, 
\begin{align}
\frac{\mu}{2} \norm{\bx-\bx^*}^2 &\le f(\bx) - f(\bx^*) \le \frac{1}{2\mu} \norm{\nabla f(\bx)}^2,\label{eq:bound_sc}\\
\frac{1}{2L} \norm{\nabla f(\bx)}^2  &\le f(\bx) - f(\bx^*) \le \frac{L}{2} \norm{\bx-\bx^*}^2, \label{eq:bound_sm}
\end{align}
where $\bx^*$ denotes the unique minimizer of $f$ on $\bbR^d$. 
\end{lemma}

\begin{lemma}\label{lem:bound_nonuniform}  
Let $f$ be defined as in $\eqref{eq:problem}$ and satisfies Assumption~\ref{assump:sc_sm}. Define a distribution $p$ with support $[n]$ such that $p_i=L_i/(n\barL)$, for any $i\in[n]$. Then for any $\bx\in\bbR^d$, 
\begin{align}
&\bbE_{i\sim p}\left[\frac{1}{np_{i}}\nabla f_{i}(\bx)\right] = \nabla f(\bx),\label{eq:expectation_nonuni}\\
&\bbE_{i\sim p}\left[\norm{\frac{1}{np_{i}}\left(\nabla f_{i}(\bx)-\nabla f_{i}(\bx^*)\right)}^2\right]\le 2\barL(f(\bx)-f(\bx^*)), \label{eq:bound_norm_nonuni}
\end{align} 
where $\bx^*$ denotes the unique minimizer of $f$ on $\bbR^d$. 
\end{lemma}

\begin{lemma}\label{lem:uniform_samp_worep} 
Let $\{\bz_i\}_{i=1}^n\subseteq\bbR^d$ and define $\barbz\defeq 1/n\sum_{i=1}^n\bz_i$. Uniformly sample a random subset $\calS$ of $[n]$ with size $b$ without replacement. Then 
\begin{equation}
\bbE_{\calS} \left[\norm{\frac{1}{b}\sum_{i\in\calS}\bz_i - \barbz}^2\right] \le \frac{n-b}{b(n-1)}\left(\frac{1}{n}\sum_{i=1}^n\norm{\bz_i}^2\right). 
\end{equation}
\end{lemma}

\bibliographystyle{IEEEtran}
\bibliography{mach_learn,math_opt,dataset,stat_ref,stoc_ref,quasi_newton,ORNMF_ref}

\end{document}